\documentclass{tran-l}

\usepackage[active]{srcltx} 

\newsymbol\ltimes 226E
\newsymbol\rtimes 226F

\newtheorem{thm}{Theorem}[section]
\newtheorem{cor}[thm]{Corollary}
\newtheorem{lem}[thm]{Lemma}
\newtheorem{prop}[thm]{Proposition}
\newtheorem{exa}[thm]{Example}

\theoremstyle{definition}
\newtheorem{defn}[thm]{Definition}
\theoremstyle{remark}
\newtheorem{rem}[thm]{Remark}
\numberwithin{equation}{section}
\newcommand{\Real}{\mathbb{R}}
\newcommand{\Complex}{\mathbb{C}}

\newcommand{\Nat}{\mathbb{N}}
\newcommand{\Rat}{\mathbb{Q}}

\newcommand{\Z}{\mathbb{Z}}
\newcommand{\Zplus}{\Z^{+}}
\newcommand{\Zmius}{\Z^{-}}
\newcommand{\Zpm}{\Z^{\pm}}

\begin{document}

\title[ Kac-Moody Lie algebras graded by .... ]
      { Kac-Moody Lie algebras graded by Kac-Moody root systems}

\author{Hechmi Ben Messaoud, Guy Rousseau}

\address{D\'epartement de Math\'ematiques,  Universit\'e de Monastir, Facult\'e des
 Sciences,  
 5019 Monastir. Tunisie.}
 \address{ Institut Elie Cartan, Universit\'e de Lorraine, BP 70239, 
  54506 Vandoeuvre L\`es-Nancy Cedex, France. }

\email{hechmi.benmessaoud@fsm.rnu.tn; Guy.Rousseau@univ-lorraine.fr}

\begin{abstract}  We look to gradations of Kac-Moody Lie algebras by Kac-Moody root systems with finite dimensional weight spaces. We  extend, to general Kac-Moody Lie algebras, the notion of $C-$admissible pair as 
introduced by H. Rubenthaler   and J. Nervi   for semi-simple and affine Lie algebras. If 
$\mathfrak{g}$ is a Kac-Moody Lie algebra (with Dynkin diagram indexed by $I$) and $(I,J)$ is  such a $C-$admissible pair, we construct a $C-$admissible subalgebra 
 $\mathfrak{g}^J$, which is a Kac-Moody Lie algebra of the same type as $\mathfrak{g}$, and  whose root system $\Sigma$ grades finitely  the Lie algebra $\mathfrak{g}$. 
 For an admissible quotient $\rho:I\to\overline I$ we build also a Kac-Moody subalgebra $\mathfrak{g}^\rho$ which grades finitely the Lie algebra $\mathfrak g$. 
 If $\mathfrak{g}$ is affine or hyperbolic, we prove that the classification of the gradations of $\mathfrak{g}$  is equivalent  to those of  the $C-$admissible pairs and of the admissible quotients.  
 For general Kac-Moody Lie algebras of indefinite type,  
 the situation may be more complicated; it is (less precisely) described by the  concept of 
 generalized $C-$admissible pairs.
  
\end{abstract}

\maketitle
\noindent
\textit{2000 Mathematics Subject Classification.} { 17B67.}

\medskip\noindent
\textit{Key words and phrases.} {\sl{ Kac-Moody algebra,  $C-$admissible pair, gradation.}}\\
\medskip
\section*{ }
{\textbf{Introduction.}} 
The notion of gradation of a Lie algebra $\mathfrak g$ by a finite root system $\Sigma$ was introduced by S. Berman and R. Moody \cite{bmo} and further studied by G. Benkart and E. Zelmanov \cite{bz}, E. Neher \cite{neh}, B. Allison, G. Benkart and Y. Gao \cite{abg} and J. Nervi \cite{nerv1}.
  This notion was extended by J. Nervi \cite{nerv2} to the case where $\mathfrak g$ is an affine Kac-Moody algebra and $\Sigma$ the (infinite) root system of an affine Kac-Moody algebra; in her two articles she uses the notion of $C-$admissible subalgebra associated to a $C-$admissible pair for the Dynkin diagram, as introduced by H. Rubenthaler \cite{rub}.
  
\par We consider here a general Kac-Moody algebra $\mathfrak g$ (indecomposable and symmetrizable) and the root system $\Sigma$ of a Kac-Moody algebra.
  We say that $\mathfrak g$ is finitely $\Sigma-$graded if $\mathfrak g$ contains a Kac-Moody subalgebra $\mathfrak m$ (the grading subalgebra) whose root system relatively to a Cartan subalgebra $\mathfrak a$ of $\mathfrak m$ is $\Sigma$ and moreover the action of $ad(\mathfrak a)$ on $\mathfrak g$ is diagonalizable with weights in $\Sigma\cup\{0\}$ and finite dimensional weight spaces, see Definition \ref{DefGrad}.
  The finite dimensionality of weight spaces is a new condition, it was fulfilled by the non trivial examples of J. Nervi \cite{nerv2} but it excludes the gradings of infinite dimensional Kac-Moody algebras by finite root systems as in \cite{bz}.
  Many examples of these gradations are provided by the almost split real forms of $\mathfrak g$, cf. \ref{ExGrad}.
  We are interested in describing the possible gradations of a given Kac-Moody algebra (as in \cite{nerv1}, \cite{nerv2}), not in determining all the Lie algebras graded by a given root system $\Sigma$ (as e.g. in \cite{abg} for $\Sigma$ finite).
  
  \par Let $I$ be the index set of the Dynkin diagram of $\mathfrak g$, we generalize the notion of $C-$admissible pair $(I,J)$ as introduced by H. Rubenthaler   \cite{rub}  and J. Nervi \cite{nerv1}, \cite{nerv2}, cf. Definition \ref{defadm}. For each Dynkin diagram $I$ the classification of the $C-$admissible pairs $(I,J)$ is easy to deduce from the list of irreducible $C-$admissible pairs due to these authors.
  We are able then to generalize in section \ref{s2} their construction of a $C-$admissible subalgebra (associated to a $C-$admissible pair) which grades finitely $\mathfrak g$:
%\begin{thm}\label{thmA} 
\medskip
\\
\textbf{Theorem 1.} (cf. \ref{AJhJ}, \ref{gJ}, \ref{jgrad}) \textit{
Let $\mathfrak g$ be an indecomposable and symmetrizable Kac-Moody algebra, associated to a generalized Cartan matrix $A=(a_{i,j})_{i,j\in I}$.
 Let $J\subset I$ be a subset of finite type such that the pair $(I,J)$ is 
 $C-$admissible. There is a generalized Cartan matrix $A^J=(a'_{k,l})_{k,l\in I'}$ with index set $I'=I\setminus J$ and a Kac-Moody subalgebra $\mathfrak{g}^J$ of $\mathfrak g$ associated to $A^J$, with root system $\Delta^J$.
 Then $\mathfrak g$  is finitely $\Delta^J-$graded with grading subalgebra $\mathfrak{g}^J$.}
%\end{thm}

\par For a general finite gradation of $\mathfrak g$ with grading subalgebra $\mathfrak m$, we prove (in section \ref{GenGrad}) that $\mathfrak m$ is also symmetrizable and the restriction to $\mathfrak m$ of the invariant bilinear form of $\mathfrak g$ is nondegenerate (Corollary \ref{Msym}).
 The Kac-Moody algebras $\mathfrak g$ and $\mathfrak m$ have the same type: finite, affine or indefinite; the first two types correspond to the cases already studied e.g. by J. Nervi.
 Moreover if $\mathfrak g$ is indefinite Lorentzian or hyperbolic, then so is $\mathfrak m$ (propositions \ref{st} and \ref{sth}).
  We get also the following precise structure result for this general situation:
%\begin{thm}\label{thmB}
\medskip 
\\
\textbf{Theorem 2.} {\textit{
Let $\mathfrak g$ be an indecomposable and symmetrizable Kac-Moody algebra, finitely graded by a root system $\Sigma$ of Kac-Moody type with grading subalgebra $\mathfrak m$.
\\
1) We may choose the Cartan subalgebras $\mathfrak a$ of $\mathfrak m$, $\mathfrak h$ of $\mathfrak g$ such that $\mathfrak a\subset \mathfrak h$. 
Then there is a surjective map $\rho_a:\Delta\cup\{0\}\to\Sigma\cup\{0\}$ between the corresponding root systems.
  We may choose the bases $\Pi_a=\{\gamma_s\mid s\in\overline I\}\subset\Sigma$ and $\Pi=\{\alpha_i\mid i\in I\}\subset\Delta$ of these root systems such that
   $\rho_a(\Delta^+)\subset\Sigma^+\cup\{0\}$ and 
   $\{\alpha\in\Delta\mid\rho_a(\alpha)=0\}=\Delta_J:=\Delta\cap(\sum_{j\in J}\,\Z\alpha_j)$
    for some subset $J\subset I$ of finite type.
    \\
2) Let $I'_{re}=\{i\in I\mid \rho_a(\alpha_i)\in\Pi_a\}$, $I'_{im}=\{i\in I\mid \rho_a(\alpha_i)\not\in\Pi_a\cup\{0\} \}$.
Then $J=\{i\in I\mid \rho_a(\alpha_i)=0\}$.
We note $I_{re}$ (resp. $J^{\circ}$)
the union of the connected components of $I\setminus I'_{im}=I'_{re}\cup J$ meeting $I'_{re}$ (resp. contained in $J$), and $J_{re}=J\cap I_{re}$.
Then the pair $(I_{re},J_{re})$ is $C-$admissible. %\marginpar{hbm-modif}
\\
3) There is a Kac-Moody subalgebra $\mathfrak g(I_{re})$ of $\mathfrak g$, associated to $I_{re}$, which contains $\mathfrak m$.
This Lie algebra is finitely $\Delta(I_{re})^{J_{re}}-$graded, with grading subalgebra $\mathfrak g(I_{re})^{J_{re}}$.
Both algebras $\mathfrak g(I_{re})$ and $\mathfrak g(I_{re})^{J_{re}}$ are finitely $\Sigma-$graded  with grading subalgebra $\mathfrak m$. 
}}
%\end{thm}

\par It may happen that $I'_{im}$ is non empty, we then say that $(I,J)$ is a generalized $C-$admissible pair. We give and explain precisely an example in section \ref{s5}.

\par When $I'_{im}$ is empty, $I_{re}=I$, $J_{re}=J$, $\mathfrak g(I_{re})=\mathfrak g$, $(I,J)=(I_{re},J_{re})$ is a $C-$admissible pair and the situation  looks much like the one described by J. Nervi in the finite \cite{nerv1} or affine  \cite{nerv2} cases.
Actually we prove that this is always true when $\mathfrak g$ is of finite type, affine or hyperbolic (Proposition \ref{imempty}).
 In this "empty" case we get the gradation of $\mathfrak g$ with two levels: $\mathfrak g$ is finitely $\Delta^J-$graded with grading subalgebra $\mathfrak g^J$ as in Theorem 1 %\ref{thmA} 
 and $\mathfrak g^J$ is finitely $\Sigma-$graded with grading subalgebra $\mathfrak m$. 
 But the gradation of $\mathfrak g^J$ by $\Sigma$ and $\mathfrak m$ is such that the corresponding set "$J$" described as in Theorem 2 %\ref{thmB} 
 is empty; we say (following \cite{nerv1}, \cite{nerv2}) that it is a maximal gradation, cf. Definition \ref{GenMaxGrad} and Proposition \ref{imempty2}.

  \par To get a complete description of the case $I'_{im}$ empty, it remains to describe the maximal gradations; this is done in section \ref{s4}.
  We prove in Proposition \ref{4.1} that a maximal gradation $(\mathfrak{g},\Sigma,\mathfrak{m})$ is entirely described by a quotient map $\rho: I\to\overline I$ which is admissible i.e. satisfies two simple conditions (MG1) and (MG2) with respect to the generalized Cartan matrix $A=(a_{i,j})_{i,j\in I}$.
  Conversely for any admissible quotient map $\rho$, it is possible to build a maximal gradation of $\mathfrak{g}$ associated to this map, cf. Proposition \ref{4.5} and Remark \ref{4.7}.
  
\section{Preliminaries}\label{s1}
We recall the basic results on the structure of  Kac-Moody Lie algebras 
and we set the notations.  More details can be found in the book of Kac \cite{vk}.
 We end by the definition of finitely graded Kac-Moody algebras.
 
\subsection{Generalized Cartan matrices}\label{1.1} Let $I$ be a finite index set. A matrix $A=(a_{i,j})_{i,j\in I}$ is called a 
\emph{generalized Cartan matrix} if it satisfies : \\
(1) $a_{i,i}=2$ $\qquad (i\in I)$\\
(2) $a_{i,j}\in\Zmius$ $\qquad (i\not= j)$\\
(3) $a_{i,j}=0 $ implies $a_{j,i}=0$.
\par
The matrix $A$ is called \emph{decomposable} if for a suitable permutation of 
$I$ it takes the form  $\left( \begin{array}{cc}
B & 0 \\ 0 & C\\ \end{array} \right)$ where $B$ and $C$ are square matrices. If $A$ is not 
decomposable, it is called \emph{indecomposable}.
\\
The matrix $A$ is called \emph{symmetrizable} if there exists an invertible diagonal matrix 
$D=\text{diag}(d_i, \; i\in I)$ such that $DA$ is symmetric. The entries $d_i, \; i\in I$,  
can be chosen to be positive rational and if moreover the matrix $A$ is indecomposable, then these
entries  are unique  up to a constant factor. 
\\
Any indecomposable generalized Cartan matrices is of one of the three 
mutually exclusive types : \emph{finite},  \emph{affine} and \emph{indefinite} (\cite{vk}, chap. 4).
\\
An indecomposable and symmetrizable generalized Cartan matrix $A$ is called \emph{Lorentzian} if it is  non-singular and the corresponding symmetric matrix has signature $(++...+-)$.\\
An indecomposable generalized Cartan matrix $A$ is called \emph{strictly hyperbolic} (resp. \emph{hyperbolic}) if the deletion of any one vertex, and the edges connected to it, of the corresponding  Dynkin diagram  yields a disjoint union of Dynkin diagrams  of finite (resp. finite or affine) type.
\\
Note that a symmetrizable  hyperbolic generalized Cartan matrix is non singular and Lorentzian (cf. \cite{moo}).% \marginpar{modif}

\subsection{Kac-Moody algebras and groups}\label{kma} (See \cite{vk} and\cite{pk}).
 \\ Let $A=(a_{i,j})_{i,j\in I}$ be an   indecomposable and symmetrizable generalized Cartan matrix. Let   
 $(\mathfrak{h}_{\Real},\;\Pi=\{\alpha_i, \, i\in I\}, \;\Pi\check{_{\,}}=
\{\alpha\check{_i}, \, i\in I\})$
 be a realization of $A$ over the real field $\Real$: thus $\mathfrak{h}_{\Real}$ is a real vector space   such that dim$\mathfrak{h}_{\Real}= |I|+\text{corank}(A)$, $\Pi$ and $\Pi\check{_{\,}}$ are linearly independent in $\mathfrak{h}_{\Real}^*$ and $\mathfrak{h}_{\Real}$ respectively such that $\langle\alpha_j,\alpha\check{_i}\rangle =
a_{i,j}$.  Let $\mathfrak{h}= \mathfrak{h}_{\Real}\otimes \Complex$, then $(\mathfrak{h},\;\Pi, \;\Pi\check{_{\,}} )$ is  a realization of $A$ over the complex field $\Complex$. 

It follows that, if $A$ is non-singular, then $\Pi\check{_{\,}}$ (resp. $\Pi$) is a basis of $\mathfrak{h}$ (resp. $\mathfrak{h}^*$); moreover $ \mathfrak{h}_{\Real}=\{h\in\mathfrak h\mid\alpha_i(h)\in\Real,\forall i\in I\}$ is well defined by the realization 
$(\mathfrak{h},\;\Pi, \;\Pi\check{_{\,}} )$. 

Let $\mathfrak{g} =\mathfrak{g}(A)$ be the complex Kac-Moody algebra associated to $A$ : it is generated by $\{\mathfrak{h}, e_i,f_i, \, i\in I\}$ with the following relations 

\begin{equation} 
\begin{array}{lll}
[\mathfrak{h}, \mathfrak{h}]=0, & [e_i,f_j]=\delta_{i,j}\alpha\check{_i} & (i,j \in I);\\
 
[h, e_i] = \langle\alpha_i, h\rangle e_i, & [h, f_i]= -\langle\alpha_i, h\rangle f_i & (h \in \mathfrak{h});\\
 
(\text{ad}e_i)^{1-a_{i,j}}(e_j)= 0, &(\text{ad}f_i)^{1-a_{i,j}}(f_j)= 0  & (i\not= j). 
\end{array}
\end{equation}
The Kac-Moody algebra $\mathfrak{g}=\mathfrak{g}(A)$ is called of finite, affine or indefinite type  if the corresponding generalized Cartan matrix $A$ is.
\par
The derived algebra $\mathfrak{g}'$ of $\mathfrak{g}$ is generated by the \emph{Chevalley generators} $e_i, f_i$, $i\in I$, and the center $\mathfrak{c}$ of $\mathfrak{g}$ lies in $\mathfrak{h}'=\mathfrak{h}\cap\mathfrak{g}'=\sum_{i\in I}\Complex\alpha\check{_i}$. 
If the generalized Cartan matrix $A$ is non-singular, %(that is the case when $A$ is of finite type or hyperbolic)
 then $\mathfrak{g}=\mathfrak{g}'$ is a (finite or infinite)-dimensional simple Lie algebra, and the center $\mathfrak{c}$ is trivial.
\par
The subalgebra $\mathfrak{h}$  is a maximal ad$(\mathfrak{g})-$diagonalizable subalgebra of $\mathfrak{g}$, it is called the \emph{standard Cartan subalgebra} of $\mathfrak{g}$. Let $\Delta=\Delta(\mathfrak{g},\mathfrak{h})$ be the corresponding root system; then $\Pi$ is a root basis of $\Delta$ and $\Delta=\Delta^+\cup\Delta^-$, where $\Delta^{\pm}=\Delta\cap\Zpm\Pi$ is the set of positive (or negative) roots relative to the basis $\Pi$. For
$\alpha\in\Delta$, let $\mathfrak{g}_{\alpha}$ be the root space of $\mathfrak{g}$ corresponding to the root $\alpha$; then $\mathfrak{g} = \displaystyle\mathfrak{h}\oplus  (\mathop{\oplus}_{\alpha\in\Delta}\mathfrak{g}_{\alpha})$.
\par
The \emph{Weyl group} $W$ of $(\mathfrak{g},\mathfrak{h})$ is generated by the fundamental reflections $r_i$ $(i\in I)$ such that $r_i(h)=h-\langle \alpha_i, h\rangle\alpha\check{_i}$ for $h\in\mathfrak{h}$, it is a Coxeter group on $\{r_i, \, i\in I\}$ with length function $w \mapsto l(w)$, $w\in W$. The Weyl group $W$ acts on $\mathfrak{h}^*$ and $\Delta$, we set $\Delta^{re}=W(\Pi)$ (the real roots) and $\Delta^{im}=\Delta\setminus\Delta^{re}$ (the imaginary roots). Any root basis of $\Delta$ is $W-$conjugate to $\Pi$ or $-\Pi$. 
\par
A \emph{Borel subalgebra} of $\mathfrak{g}$ is a maximal completely solvable subalgebra. A \emph{parabolic subalgebra} of $\mathfrak{g}$ is a (proper) subalgebra containing a Borel subalgebra. 
The \emph{standard positive (or negative) Borel subalgebra} is $\mathfrak{b}^{\pm}:=\mathfrak{h}\oplus (\oplus_{\alpha\in\Delta^{\pm}}\mathfrak{g}_{\alpha})$. A parabolic subalgebra $\mathfrak{p}^+$ (resp. $\mathfrak{p}^-$) containing $\mathfrak{b}^+$ (resp. $\mathfrak{b}^-$) is called \emph{positive (resp. negative) standard parabolic subalgebra} of 
$\mathfrak{g}$; then there exists a subset $J$ of $I$ (called the type of $\mathfrak{p}^{\pm}$) such that  $\mathfrak{p}^{\pm}=\mathfrak{p}^{\pm}(J):=\displaystyle(\mathop{\oplus}_{\alpha\in \Delta_J}\mathfrak{g}_{\alpha}) +\mathfrak{b}^{\pm}$, where $\Delta_J=\Delta\cap(\oplus_{j\in J}\Z \alpha_j)$ (cf. \cite{kw}).
\par
In \cite{pk}, D.H. Peterson and V.G. Kac construct a group $G$, which is the connected and simply connected complex algebraic group associated to $\mathfrak{g}$ when $\mathfrak{g}$ is of finite type, depending only on the derived Lie algebra $\mathfrak{g}'$ and acting on $\mathfrak{g}$ via the adjoint representation  Ad : $G \rightarrow
\text{Aut}(\mathfrak{g})$. It is generated by the one-parameter subgroups $U_{\alpha}=\text{exp}(\mathfrak{g}_{\alpha})$, $\alpha\in\Delta^{re}$, and 
Ad$(U_{\alpha})= \text{exp(ad}\mathfrak{g}_{\alpha}))$. 
 In the definitions of J. Tits \cite{tits} $G$ is the group of complex points of $\mathfrak G_D$ where $D$ is the datum associated to $A$ and the $\Z-$dual $\Lambda$ of $\bigoplus_{i\in I}\Z\alpha_i^\vee$.
\par
The Cartan subalgebras of $\mathfrak{g}$ are $G-$conjugate.  If $\mathfrak{g}$ is not of finite type, there are exactly two conjugate classes (under the adjoint action of $G$) of Borel subalgebras : $G.\mathfrak{b}^+$ and $G.\mathfrak{b}^-$. A Borel subalgebra $\mathfrak{b}$ of $\mathfrak{g}$ which is  $G-$conjugate to $\mathfrak{b}^+$ (resp. $\mathfrak{b}^-$) is called positive (resp. negative). It follows that any parabolic subalgebra $\mathfrak{p}$ of $\mathfrak{g}$ is  $G-$conjugate to a standard positive (or negative) parabolic subalgebra, in which case,  we say that $\mathfrak{p}$ is positive (or negative).

\subsection{Standard Kac-Moody subalgebras and subgroups}
Let  $J$ be a nonempty subset of $I$. Consider the generalized Cartan matrix $A_J= (a_{i,j})_{i,j\in J}$.
\begin{defn} 
The subset $J$ is called of finite, affine or indefinite type if the corresponding generalized Cartan matrix $A_J$ is.  We say also that $J$ is connected, if the Dynkin subdiagram, with  vertices indexed by $J$, is connected, or equivalently, the corresponding generalized Cartan submatrix $A_J$ is indecomposable. 
\end{defn}
\begin{prop}\label{realj}
Let $\Pi_J=\{\alpha_j, \; j\in J\}$ and $\Pi\check{_J}=
\{\alpha\check{_j}, \; j\in J\}$. Let $\mathfrak{h}_J'$ be the subspace of $\mathfrak{h}$ generated by $\Pi\check{_J}$, and  $\mathfrak{h}^J=\Pi_J^{\perp}=\{h\in\mathfrak{h}, \langle\alpha_j,h\rangle=0, \; \forall j\in J\}$. Let $\mathfrak{h}_J''$ be a supplementary subspace of $\mathfrak{h}_J'+ \mathfrak{h}^J$ in $\mathfrak{h}$ 
 and let $$\mathfrak{h}_J= \mathfrak{h}_J' \oplus \mathfrak{h}_J'',$$ then, we have :\\
1)  $(\mathfrak{h}_J,\;\Pi_J, \;\Pi\check{_J} )$ is  a  realization of the generalized Cartan matrix $A_J$. 
Hence $\mathfrak{h}_J''=\{0\}$, $\mathfrak{h}_J=\mathfrak{h}_J'$ when $A_J$ is regular (e.g. when $J$ is of finite type).
\\
2) The subalgebra $\mathfrak{g}(J)$ of $\mathfrak{g}$, generated by $\mathfrak{h}_J$ and the $e_j$, $f_j$, $j\in J$, is the Kac-Moody Lie algebra associated to the realization  $(\mathfrak{h}_J,\;\Pi_J, \;\Pi\check{_J} )$ of  $A_J$. \\
3) The corresponding  root system  $\Delta(J)=\Delta(\mathfrak{g}(J),\mathfrak{h}_J)$ can be identified with $\Delta_J:=\Delta\cap(\oplus_{j\in J}\Z \alpha_j)$.
\end{prop}
\textbf{N.B.} The derived algebra $\mathfrak{g}'(J)$ of $\mathfrak{g}(J)$ is generated by the $e_j,f_j$ for $j\in J$; it does not depend of the choice of $\mathfrak{h}_J''$.
\begin{proof}$\,$\\
1) Note that $\dim(\mathfrak{h}_J'') =\dim(\mathfrak{h}_J'\cap\mathfrak{h}^J)=corank(A_J)$. In particular,  $\dim(\mathfrak{h}_J)-
|J|= corank(A_J)$. If $\alpha\in Vect(\alpha_j, \; j\in J)$, then $\alpha$ is entirely determined by its restriction to   $\mathfrak{h}_J$ and hence $\Pi_J$ defines, by restriction, a free family in $\mathfrak{h}_J^*$. As $\Pi\check{_J} $ is  free,  assertion 1) holds.  \\
Assertions 2) and 3) are straightforward.
\end{proof}
In the same way, the subgroup $G_J$ of $G$ generated by $U_{\pm\alpha_j}$, $j\in J$, is equal to  the Kac-Moody group associated to the generalized Cartan matrix $A_J$: it is clearly a quotient; the equality is proven in \cite[5.15.2]{rou}.
\subsection{The invariant bilinear form}\label{ibf} (See \cite{vk}).\\ We recall that the generalized Cartan matrix $A$ is supposed indecomposable and symmetrizable. There exists a nondegenerate ad$(\mathfrak{g})-$ invariant symmetric $\Complex-$bilinear form 
$( . \, , \, .)$ on $\mathfrak{g}$, which is entirely determined by its restriction to $\mathfrak{h}$, such that $$(\alpha\check{_i},h)=\frac{(\alpha\check{_i},\alpha\check{_i})}{2}\langle\alpha_i, h\rangle, \quad  i\in I, \; h\in\mathfrak{h},$$ 
and we may thus assume that \begin{equation}\label{bf}
(\alpha\check{_i},\alpha\check{_i}) \; \text{is a positive rational for all}\; i.
\end{equation} The nondegenerate  invariant bilinear form $( . \, , \, .)$ induces an isomorphism $\nu$ : $\mathfrak{h}\to\mathfrak{h}^*$ such that $\alpha_i= \displaystyle \frac{2\nu(\alpha\check{_i})}{(\alpha\check{_i},\alpha\check{_i})}$  and  $\alpha\check{_i}= \displaystyle \frac{2\nu^{-1}(\alpha{_i})}{(\alpha{_i},\alpha{_i})}$ for all $i$.\\
There exists a totally isotropic subspace $\mathfrak{h}''$ of $\mathfrak{h}$ (relatively to  $( . \, , \, .)$) which is in duality  with the center $\mathfrak{c}$ of $\mathfrak{g}$. In particular, $\mathfrak{h}''$ defines a supplementary subspace of $\mathfrak{h}'$ in $\mathfrak{h}$.\\
Note that  any invariant  symmetric bilinear form $b$ on $\mathfrak{g}$ satisfying  $b(\alpha\check{_i},\alpha\check{_i}) >0$, 
$\forall i\in I$, is nondegenerate and $b(\alpha\check{_i},h)=\frac{b(\alpha\check{_i},\alpha\check{_i})}{2}\langle\alpha_i, h\rangle, \; \forall  i\in I, \; \forall h\in\mathfrak{h}$. It follows that the restriction of $b$ to $\mathfrak{g}'$ is proportional to that of $( . \, , \, .)$.
In particular, if  $A$ is  non-singular, then the invariant bilinear form $( . \, , \, .)$  satisfying the condition \ref{bf} is unique up to a positive rational factor. 

\subsection{The Tits cone}(See \cite{vk}, chap. 3 and 5).\\ Let $C:=\{h\in\mathfrak{h}_{\Real}; \langle\alpha_i,h\rangle\geq 0, \forall  i\in I\} $ be the fundamental chamber (relative to the root basis $\Pi$) and let $X:=\displaystyle\mathop{\bigcup}_{w\in W}w(C)$ be the Tits cone. We have the following description of the Tits cone: \\
(1) $X=\{h\in\mathfrak{h}_{\Real}; \langle\alpha, h\rangle <0$ only for a finite number of $\alpha\in\Delta^+\}$.\\
(2) $X=\mathfrak{h}_{\Real}$ if and only if the generalized Cartan matrix $A$ is of finite type.\\
(3) If $A$ is of affine type, then $X=\{h\in\mathfrak{h}_{\Real}; \langle\delta, h\rangle >0\}\cup\Real\nu^{-1}(\delta)$, where $\delta$ is the lowest imaginary positive root of $\Delta^+$. \\
(4) If $A$ is of indefinite  type, then the closure of the Tits cone, for the metric topology on $\mathfrak{h}_{\Real}$,  is ${\bar X}=\{h\in\mathfrak{h}_{\Real}; \langle\alpha, h\rangle \geq 0, \; \forall\alpha\in\Delta_{im}^+ \}$. \\
(5) If $h\in X$, then $h$ lies in the interior $\buildrel\circ\over X$ of $X$ if and only if the fixator $W_h$ of $h$, in the Weyl group $W$, is finite. Thus $\buildrel\circ\over X$ is the union of finite type facets of $X$. \\
(6) If $A$ is hyperbolic, then ${\bar X}\cup(-{\bar X})=\{h\in\mathfrak{h}_{\Real}; (h, h) \leq 0\}$ and the set of imaginary roots is $\Delta^{im}=\{\alpha\in Q\setminus\{0\}; (\alpha,\alpha)\leq 0 \}$, where $Q=\Z\Pi$ is the root lattice. 

\begin{rem}\label{tc}
Combining  (3) and (4) one obtains that if $A$ is not of finite type then  ${\bar X}=\{h\in\mathfrak{h}_{\Real}; \langle\alpha, h\rangle \geq 0, \; \forall\alpha\in\Delta_{im}^+ \}$. 
\end{rem}

\subsection{Graded Kac-Moody Lie algebras}
From now on we suppose that the  Kac-Moody Lie algebra $\mathfrak{g}$ is indecomposable and symmetrizable.

\begin{defn}\label{DefGrad} Let $\Sigma$ be a root system of Kac-Moody type. The Kac-Moody Lie algebra $\mathfrak{g}$ is said to be finitely  $\Sigma-$graded if : \\
(i) $\mathfrak{g}$ contains, as  a subalgebra,  a Kac-Moody algebra ${\mathfrak{m}}$ whose root system relative to Cartan subalgebra ${\mathfrak{a}}$ is equal to $\Sigma$.\\
(ii) $\mathfrak{g}=\displaystyle\sum_{{\alpha}\in\Sigma \cup\{0\}}V_{{\alpha}}$, with $V_{{\alpha}}=\{x\in \mathfrak{g}\, ; \;
[a, x]=\langle {\alpha}, a\rangle x, \; \forall a \in{\mathfrak{a}}\}$. \\
(iii) $V_{{\alpha}}$ is finite dimensional  for all ${{\alpha}\in\Sigma\cup\{0\}}$.
\par We say that $\mathfrak m$ (as in (i) above) is a grading subalgebra, and $(\mathfrak g,\Sigma,\mathfrak m)$ a gradation with finite multiplicities (or, to be short, a finite gradation). 
\end{defn}

Note that from (ii) the Cartan subalgebra ${\mathfrak{a}}$ of ${\mathfrak{m}}$ is ad$(\mathfrak{g})-$diagonalizable, and we may assume that ${\mathfrak{a}}$ is contained in the standard  Cartan subalgebra $\mathfrak{h}$ of $\mathfrak{g}$.

\begin{lem}\label{trans}Let $\mathfrak{g}$ be a Kac-Moody algebra finitely  $\Sigma-$graded, with grading subalgebra $\mathfrak m$. 
If $\mathfrak m$ itself is finitely $\Sigma'-$graded (for some root system $\Sigma'$of Kac-Moody type), then $\mathfrak{g}$ is finitely $\Sigma'-$graded.
\end{lem}
\begin{proof} If $\mathfrak m'$ is the grading subalgebra of $\mathfrak m$, we may suppose the Cartan subalgebras such that $\mathfrak a'\subset\mathfrak a\subset\mathfrak h$, with obvious notations.
 Conditions (i) and (ii) are clearly satisfied for $\mathfrak g$, $\mathfrak m'$ and $\mathfrak a'$.
 Condition (iii) for $\mathfrak m$ and $\Sigma'$ tells that, for all $\alpha'\in\Sigma'$, the set 
 $\{\alpha\in\Sigma\mid\alpha_{\vert\mathfrak a'}=\alpha'\}$ is finite.
  But $V_{\alpha'}=\oplus_{\alpha_{\vert\mathfrak a'}=\alpha'}\,V_{\alpha}$, so each $V_{\alpha'}$ is finite dimensional if this is true for each $V_{\alpha}$.
\end{proof}

\subsection{Examples of gradations.}\label{ExGrad} $\,$\\  
1) Let $\Delta= \Delta(\mathfrak{g},\mathfrak{h})$ the root system of $\mathfrak{g}$ relative to $\mathfrak{h}$, then $\mathfrak{g}$ is finitely $\Delta-$graded : this is the trivial gradation of $\mathfrak{g}$ by its own  
root system.\\
2) Let $\mathfrak{g}_{\Real}$ be an almost split real form of $\mathfrak{g}$ (see \cite{b3r}) and let $\mathfrak{t}_{\Real}$ be a maximal split toral subalgebra of $\mathfrak{g}_{\Real}$. Suppose that the restricted root system  $\Delta'= \Delta(\mathfrak{g}_{\Real},\mathfrak{t}_{\Real})$ is reduced of Kac-Moody type. In [\cite{ba}, \S 9], N. Bardy constructed a split real Kac-Moody subalgebra  $\mathfrak{l}_{\Real}$ of $\mathfrak{g}_{\Real}$ such that $\Delta'= \Delta(\mathfrak{l}_{\Real},\mathfrak{t}_{\Real})$, then $\mathfrak{g}$ is obviously  finitely $\Delta'- $graded.

\par We get thus many examples coming from known tables for almost split real forms: see \cite{b3r} in the affine case and \cite{bm} in the hyperbolic case.
\\
3) When $\mathfrak{g}_{\Real}$ is an almost compact real form of $\mathfrak{g}$, the same constructions should lead to gradations by finite root systems, as in \cite{bz} e.g..
\section{Gradations associated to $C-$admissible pairs.}\label{s2}$\,$\\
We recall some definitions introduced by H. Rubenthaler (\cite{rub})  and J. Nervi (\cite{nerv1}, \cite{nerv2}).  
Let $J$ be a subset of $I$ of finite type. For $k\in I\setminus J$, we denote by $I_k$ the connected component, containing $k$, of the Dynkin subdiagram corresponding to $J\cup\{k\}$, and let $J_k:=I_k\setminus\{k\}$. 

 Suppose now that $I_k$ is of finite type for all $k\in I\setminus J$ : that is always the case if $\mathfrak{g}$ is of affine type and $ |I\setminus J|\geq 2$ or $\mathfrak{g}$ is of hyperbolic type and $ |I\setminus J|\geq 3$. \\
For $k\in I\setminus J$, let $\mathfrak{g}(I_k)$ be the simple subalgebra generated by $\mathfrak{g}_{\pm\alpha_i}$, $i\in I_k$, then  $\mathfrak{h}_{I_k}=\mathfrak{h}\cap\mathfrak{g}(I_k)= \sum_{i\in I_k}\Complex\alpha\check{_i}$ is a Cartan subalgebra of $\mathfrak{g}(I_k)$. Let $H_k$ be the unique element of $\mathfrak{h}_{I_k}$ such that $\langle\alpha_i, H_k\rangle =2\delta_{i,k}$, $\forall i\in I_k$.

\begin{defn}\label{defadm} We  preserve the notations and the assumptions introduced above. \\ 
1) Let $k\in I\setminus J$.  \\
(i) The pair $(I_k,J_k)$ is called admissible if there exist $E_k, F_k\in \mathfrak{g}(I_k)$ such that  $(E_k,H_k,F_k)$ is an $\mathfrak{sl}_2-$triple.\\
(ii) The pair $(I_k,J_k)$ is called $C-$admissible if it is admissible and the simple Lie algebra $\mathfrak{g}(I_k)$ is $A_1-$graded 
by the root system of type $A_1$ associated to the $\mathfrak{sl}_2-$triple $(E_k,H_k,F_k)$.
\\
2) the pair $(I,J)$  is called $C-$admissible if the pairs $(I_k,J_k)$ are $C-$admissible for all $k\in I\setminus J$.
\end{defn}

Schematically, any $C-$admissible pair $(I,J)$ is represented  by the Dynkin diagram, corresponding to $A$, on which the vertices indexed by   $J$   are denoted by white circles $\circ$ and those of $I\setminus J$ are denoted by black circles $\bullet$.

\par  It is known that, when $(I_k,J_k)$ is admissible,  $H_k=\sum_{i\in I_k}n_{i,k}\alpha\check{_i}$, where $n_{i,k}$ are positive integers (see \cite{rub} or [\cite{nerv2}; Prop. 1.4.1.2]). 

\begin{rem}\label{eqc}
Note that this definition, for $C-$admissible pairs, is equivalent to that introduced by Rubenthaler and Nervi (see \cite{rub}, \cite{nerv1}) in terms of  prehomogeneous spaces of parabolic type :  if $(I_k,J_k)$ is $C-$admissible, define   for  $p\in\Z$,  the subspace $d_{k,p}:=\{X\in \mathfrak{g}(I_k) \, ;\; [H_k, X]=2pX\}$; then  $(d_{k,0},d_{k,1})$ is an irreducible regular and commutative  prehomogeneous space of parabolic type, and  $d_{k,p}=\{0\}$ for $|p|\geq 2$. Then  we say that $(I_k,J_k)$ is  an irreducible $C-$admissible pair. According to Rubenthaler and  Nervi ([\cite{rub}; table1] or [\cite{nerv1}; table 2] ) the irreducible $C-$admissible pair $(I_k,J_k)$ should be among the  list in Table 1 below.%\marginpar{hbm-modif}
\end{rem}

\begin{defn}\label{jc} Let $J$ be a subset of $I$ and let $i,k\in I\setminus J$.  We say that $i$ and $k$ are $J-$connected relative to $A$ if there exist $j_0,j_1, ...., j_{p+1}\in I$ such that $j_0=i$, $j_{p+1}=k$, $j_s\in J$, $\forall s=1, 2, ..., p$, and $a_{j_s, j_{s+1}}\not=0$, $\forall s=0,1, ..., p$. 
\end{defn} 
\begin{rem}\label{rjc}
Note that the relation  `` to be $J-$connected '' is symmetric on $i$ and $k$. As the generalized Cartan matrix is assumed to be indecomposable, for any  vertices $i, k\in I\setminus J$ there exist  $i_0, i_1, ...., i_{p+1}\in I\setminus J$ such that $i_0=i$, $i_{p+1}=k$ and  $i_s$ and $i_{s+1}$ are $J-$connected  for all $s=0, 1, ..., p$.
\end{rem}

\medskip
\begin{center}
 {\textbf{Table 1 }}
\end{center}
\medskip
\begin{center}
{{List of irreducible $C-$admissible pairs }}
\end{center}
\medskip
\begin{center}
\begin{tabular}{|l|l|}
\hline
  &  \\
$A_{2n-1}$, ${_{n\geq 1}}$  &
\hbox to 5cm{\lower 2pt \hbox{$\displaystyle\mathop{\circ}^{_1}$}
\hglue -4pt \hrulefill\lower 2pt
\hbox{\hglue -0pt$\displaystyle\mathop{\circ}^{_2}$}
%\hbox{\hglue -2pt${\bullet}$}
\hglue -4pt\hrulefill\lower 0.5pt \hbox{....}
\hglue -2pt \hrulefill\hglue -4pt\hrulefill\hglue 0pt\lower 2pt \hbox{$\circ$}
\hglue -4pt \hrulefill\lower 2pt
\hbox{\hglue -0pt$\displaystyle\mathop{\bullet}^{_n}$}
\hglue -4pt\hrulefill\hglue 0pt\lower 2pt \hbox{$\circ$}\hglue -0pt
\hglue -1pt\hrulefill\hglue 0pt\lower 0.5pt \hbox{....}
\hglue -2pt \hrulefill
\lower 2pt \hbox{$\circ$}\hglue -0pt \hrulefill
\lower 2pt \hbox{\hglue -6pt$\displaystyle\mathop{\circ}^{_{2n-1}}$}  }  \\
 & \\
 \hline
 & \\
$B_n$, ${_{n\geq 3}}$ &
\hbox to 4cm {\lower 2pt\hbox{$\displaystyle
\mathop{\bullet}^{_1}$}\hglue -1,6pt
 \hrulefill\lower 2pt\hbox{$\displaystyle
\mathop{\circ}^{_2}$}\hglue -1pt\hrulefill\lower 2pt\hbox{$\displaystyle
\mathop{\circ}^{_3}$}\hglue -1pt\hrulefill\lower 0.5pt
\hbox{......}\lower 2pt \hbox{}\hglue -0pt \hrulefill
\lower 2pt\hbox{$\displaystyle\mathop{\circ}$}\hglue -1pt\hrulefill
\lower 2pt
\hbox{\offinterlineskip {$\circ$}\hglue -2,8pt\vbox{ {\hrule height 0,3pt
width 0,8cm}\vskip 3pt{\hrule height 0,3pt width 0,8cm}
\vskip 0,9pt}\hglue -11pt $>$ \hglue -2,7pt{$\displaystyle\mathop{\circ}^{_n}$
}}}
\\
 & \\
 \hline
 & \\
$C_n$, ${_{n\geq 2}}$ &
\hbox to 4cm {\lower 2pt\hbox{$\displaystyle
\mathop{\circ}^{_1}$}\hglue -1,6pt
 \hrulefill\lower 2pt\hbox{$\displaystyle
\mathop{\circ}^{_2}$}\hglue -1pt\hrulefill\lower 2pt\hbox{$\displaystyle
\mathop{\circ}^{_3}$}\hglue -1pt\hrulefill\lower 0.5pt
\hbox{......}\lower 2pt \hbox{}\hglue -0pt \hrulefill
\lower 2pt\hbox{$\displaystyle\mathop{\circ}$}\hglue -1pt\hrulefill
\lower 2pt
\hbox{\offinterlineskip {$\circ$}\hglue -2,8pt\vbox{ {\hrule height 0,3pt
width 0,8cm}\vskip 3pt{\hrule height 0,3pt width 0,8cm}
\vskip 0,9pt}\hglue -15pt $<$ \hglue +1,5pt{$\displaystyle\mathop{\bullet}^{_n}$
}}}
\\
& \\
\hline
& \\
$D_{n,1}$, $_{{n\geq 4}}$ &  
\hbox to 4cm{\lower 2pt \hbox{$\displaystyle \mathop{\bullet}^{_1}$}
\hglue -4pt\hrulefill\lower 2pt
\hbox{$\displaystyle\mathop{\circ}^{_2}$}
\hglue -4pt\hrulefill\hbox{$....$}\hglue -0pt\hrulefill\lower 2pt\hbox{$\displaystyle\mathop{\circ}$}\hglue -0pt
\hrulefill\lower 2pt\vtop{\offinterlineskip \hbox{$\displaystyle\mathop{\circ}$}
\hbox to 5pt{\hfill \vrule height 6pt width 0,3pt \hfill}
\hbox{$\displaystyle\mathop{\circ}_{^n}$} } \hglue -4pt\hrulefill\lower 2pt
\hbox{\hrulefill\hglue -4.5pt$\displaystyle\mathop{\circ}^{_{n-1}}$} }\\
%& \\
\hline
& \\
$D_{2n,2}$, $_{{n\geq 2}}$ &  
\hbox to 4cm{\lower 2pt \hbox{$\displaystyle \mathop{\circ}^{_1}$}
\hglue -4pt\hrulefill\lower 2pt
\hbox{$\displaystyle\mathop{\circ}^{_2}$}
\hglue -4pt\hrulefill\hbox{$....$}\hglue -0pt\hrulefill\lower 2pt\hbox{$\displaystyle\mathop{\circ}$}\hglue -0pt
\hrulefill\lower 2pt\vtop{\offinterlineskip \hbox to 8pt{\hfill$\displaystyle\mathop{\circ}$\hfill}
\hbox to 8pt{\hfill\vrule height 9pt width 0,3pt \hfill}
\hbox{$\displaystyle\mathop{\bullet}_{^{2n}}\;$} } \hglue -9pt\hrulefill\lower 2pt
\hbox{\hrulefill\hglue -6pt$\displaystyle\mathop{\circ}^{_{2n-1}}$} }\\
 %& \\
 \hline
 & \\
 $E_7$ &  
\hbox to 4cm{\lower 2pt\hbox{$\displaystyle\mathop{\circ}^{_1}$}
\hglue -4pt\hrulefill\lower 2pt\hbox{$\displaystyle
\mathop{\circ}^{_3}$} \hglue -4pt\hrulefill\lower 2pt
\vtop{\offinterlineskip \hbox{$\displaystyle\mathop{\circ}^{_4}$}
\hbox to 5pt{\hfill\vrule height 10pt width 0,3pt\hfill}
\hbox{$\displaystyle\mathop{\circ}_{^2}$\hfill} }
\hglue -4pt\hrulefill\lower 2pt\hbox{$\displaystyle
\mathop{\circ}^{_5}$} \hglue -4pt\hrulefill\lower 2pt
\hbox{$\displaystyle \mathop{\circ}^{_6}$} \hglue -4pt\hrulefill\lower 2pt
\hbox{$\displaystyle \mathop{\bullet}^{_7}$}}
 \\
\hline
\end{tabular}
\end{center}

\medskip

Let us assume from now on that $(I,J)$ is a $C-$admissible pair and let $I':=I\setminus J$. For $k\in I'$, let $(E_k,H_k,F_k)$ be an $\mathfrak{sl}_2-$triple associated to the irreducible $C-$admissible pair $(I_k,J_k)$.  

\begin{lem}\label{ljc} Let $k\not=l\in I'$, then : \\
1) $\langle \alpha_l, H_k\rangle\in\Zmius$.\\
2) the following assertions are equivalent : 

i) $k,l$ are $J-$connected

ii)  $\langle \alpha_l, H_k\rangle$ is a negative integer 

iii)  $\langle \alpha_k, H_l\rangle$ is a negative integer
\end{lem}
\begin{proof} $\,$\\
1) Recall that   $H_k=\sum_{i\in I_k}n_{i,k}\alpha\check{_i}$, where $n_{i,k}$ are positive integers. As $l\notin I_k$, we have that  $\langle \alpha_l, H_k\rangle=\sum_{i\in I_k}n_{i,k}\langle \alpha_l, \alpha\check{_i}\rangle\in\Zmius$.\\
2) In view of Remark \ref{rjc}, it suffices to prove the equivalence between i) and ii).  Since $I_k$ is the connected component of $J\cup\{k\}$ containing $k$, the assertion i)  is equivalent to say that the vertex $l$ is connected to $I_k$, so there exists $i_k\in I_k$  such that  $\langle \alpha_l, \alpha\check{_{i_k}}\rangle <0$ and hence $\langle \alpha_l, H_k\rangle <0$.
\end{proof}

\begin{prop}\label{AJhJ} Let $\mathfrak{h}^J=\Pi_J^{\perp}=\{h\in\mathfrak{h}, \langle\alpha_j,h\rangle=0, \; \forall j\in J\}$. For $k\in I'$, denote by $\alpha_k'=\alpha_k/\mathfrak{h}^J$ the restriction of $ \alpha_k$ to the subspace $\mathfrak{h}^J$ of $\mathfrak{h}$, and $\Pi^J=\{\alpha'_k; \; k\in I'\}$, $\Pi^{J\vee}=\{H_k; \; k\in I'\}$. For $k,l\in I'$, put $a'_{k,l} =\langle \alpha_l, H_k\rangle$  and  $A^J=(a'_{k,l})_{k,l\in I'}$. Then $A^J$ is an indecomposable and symmetrizable generalized Cartan matrix,  $(\mathfrak{h}^J, \Pi^J, \Pi^{J\vee})$ is a realization of $A^J$ and corank$(A^J)=$ corank$(A)$.
\end{prop}
\begin{proof}
The fact that $a'_{k,k}=2$ follows from the definition of $H_k$ for $k\in I'$. If $k\not=l\in I'$, then by lemma \ref{ljc},  $a'_{k,l}\in\Zmius$ and $a'_{k,l}\not=0$ if and only if  $a'_{l,k}\not=0$. Hence $A^J$ is a generalized Cartan matrix. As the matrix $A$ is  indecomposable,  $A_J$ is also indecomposable  (see Remark \ref{rjc}). 
Clearly $\Pi^J=\{\alpha'_k; \; k\in I'\}$ is free in ${\mathfrak{h}^J}^*$ the dual space of $\mathfrak{h}^J$,  $\Pi^{J\vee}=\{H_k; \; k\in I'\}$ is free in ${\mathfrak{h}^J}$ and by construction $\langle \alpha_l, H_k\rangle=a'_{k,l}$, $\forall k,l\in I'$. \\
We have to prove that 
$\dim(\mathfrak{h}^J)-|I'|=corank(A^J)$. As $J$ is of finite type, the restriction of the invariant bilinear form $( . \, , \, .)$ to  ${\mathfrak{h}_J}$ is nondegenerate and  ${\mathfrak{h}_J}$ is contained in   ${\mathfrak{h}'}=\displaystyle\mathop{\oplus}_{i\in I}\Complex \alpha\check{_i}$. Therefore   $${\mathfrak{h}}={\mathfrak{h}^J}\displaystyle\mathop{\oplus}^{\perp}{\mathfrak{h}_J} $$
and $${\mathfrak{h}'}=({\mathfrak{h}'}\cap{\mathfrak{h}^J})\oplus{\mathfrak{h}_J}.$$
It follows that  
$\dim({\mathfrak{h}'}\cap{\mathfrak{h}^J})=|I'|=\dim(\displaystyle\mathop{\oplus}_{k\in I'}\Complex H_k)$. As the subspace $\displaystyle\mathop{\oplus}_{k\in I'}\Complex H_k$ is  contained in ${\mathfrak{h}'}\cap{\mathfrak{h}^J}$, we deduce that ${\mathfrak{h}'}\cap{\mathfrak{h}^J}=\displaystyle\mathop{\oplus}_{k\in I'}\Complex H_k$. Note that any supplementary subspace ${\mathfrak{h}^J}''$ of  ${\mathfrak{h}'}\cap{\mathfrak{h}^J}$ in ${\mathfrak{h}^J}$ is also a supplementary of   ${\mathfrak{h}'}$ in ${\mathfrak{h}}$; hence, we have that  $corank(A)=\dim({\mathfrak{h}^J}'')=\dim({\mathfrak{h}^J})-|I'|$. In addition, it is known that $corank(A)=\dim({\mathfrak{c}})$, where ${\mathfrak{c}}=\displaystyle\mathop{\cap}_{i\in I}\ker(\alpha_i)$ is the center of ${\mathfrak{g}}$. Let ${\mathfrak{c}^J}=\displaystyle\mathop{\cap}_{k\in I'}\ker(\alpha_k')$, then  ${\mathfrak{c}^J}={\mathfrak{c}}$ and $corank(A^J)=\dim({\mathfrak{c}^J})=corank(A)=\dim(\mathfrak{h}^J)-|I'|$.\\
It remains to prove that $A^J$ is symmetrizable. 
For $k\in I'$, let $R_k^J$ be the fundamental reflection of $\mathfrak{h}^J$ such that $R_k^J(h)=h-\langle\alpha_k',h\rangle H_k$, $\forall h\in \mathfrak{h}^J$. Let $W^J$ be the Weyl group of $A^J$ generated by $R_k^J$, $k\in I'$.  Let $( . \, , \, .)^J$ be  the restriction to $\mathfrak{h}^J$ of the invariant bilinear form $( . \, , \, .)$ on  $\mathfrak{h}$. Then $( . \, , \, .)^J$ is a 
nondegenerate symmetric bilinear form on  $\mathfrak{h}^J$ which is $W^J-$invariant (see the lemma hereafter). From the  relation $(R_k^J(H_k),R_k^J(H_l))^J= (H_k,H_l)^J$ one can deduce that :  $$(H_k,H_l)^J= \frac{(H_k,H_k)^J}{2}a'_{l,k}, \; \forall k,l\in I',$$
but $(H_k,H_k)^J>0$, $\forall k\in I'$;  hence  $^tA^J$  (and so $A^J$) is symmetrizable. 
\end{proof}
\begin{lem} For $k\in I':=I\setminus J$, let $w_k^J$ be the longest element of the Weyl group $W(I_k)$ generated by the fundamental reflections $r_i$, $i\in I_k$. Then $w_k^J$ stabilizes  $\mathfrak{h}^J$ and induces the fundamental reflection $R_k^J$ of $\mathfrak{h}^J$ associated to $H_k$.
\end{lem}
\begin{proof} If one looks at the list above of the irreducible $C-$admissible  pairs, one can see that $w_k^J(\alpha_k)=-\alpha_k$ and that 
$-w_k^J$ permutes the $\alpha_j$, $j\in J_k$. In addition,  $w_k^J(\alpha_j)=\alpha_j$, $\forall j\in J\setminus J_k$. Now it's clear that $w_k^J$ stabilizes  $\mathfrak{h}_J$ and hence it stabilizes $\mathfrak{h}^J=\mathfrak{h}_J^{\perp}$. Note that $-w_k^J(H_k)\in \mathfrak{h}_{I_k}$ and  satisfies the same equations defining $H_k$. Hence $-w_k^J(H_k)=H_k=-R_k^J(H_k)$. Clearly $w_k^J$ and $R_k^J$ fix both $\ker(\alpha_k')=\ker(\alpha_k)\cap(\displaystyle\mathop{\cap}_{j\in J}\ker(\alpha_j))$.  As $\mathfrak{h}^J=\ker(\alpha_k')\oplus\Complex H_k$, the reflection $R_k^J$  and $W_k^J$ coincide  on $\mathfrak{h}^J$.
\end{proof}\noindent
\begin{rem}\label{remadm}
Actually we can now rediscover the list of irreducible $C-$admissible pairs given in Remark \ref{eqc}. The black vertex $k$ should be invariant under $-w_k^J$ and the corresponding coefficient of the highest root of $I_k$ should be $1$ (an easy consequence of the definition \ref{defadm} 1) (ii) ).
\end{rem}
\begin{exa} Consider the hyperbolic generalized Cartan matrix $A$ of type $HE_{8}^{(1)}=E_{10}$ indexed by $I=\{-1,0,1,..., 8\}$. \\The  following two choices  for $J$ define   $C-$admissible pairs :  \\
1) $J=\{2,3,4,5\}$. 
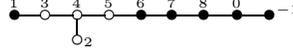
\begin{center}
\hskip 3cm \hbox to 4cm{\lower 2pt
\hbox{\hglue -1pt$\displaystyle\mathop{\bullet}^{_1}$}
\hglue -4.5pt\hrulefill\lower 2pt\hbox{\hglue -1pt$\displaystyle
\mathop{\circ}^{_3}$} \hglue -4pt\hrulefill\lower 2pt
\hbox{\hglue -1pt
\vtop{ \offinterlineskip \hbox{$\displaystyle\mathop{\circ}^{_4}$}
\vskip -1pt
\hbox to 5pt{\hfill\vrule height 6pt width 0,3pt\hfill}
\hbox{${\circ} _{^2}$\hfill}
} }
\hglue -12pt\hrulefill\lower 2pt\hbox{\hglue -1pt$\displaystyle
\mathop{\circ}^{_5}$} \hglue -4pt\hrulefill\lower 2pt
\hbox{\hglue -1pt$\displaystyle \mathop{\bullet}^{_6}$}\hglue -1pt\hrulefill\lower 2pt
\hbox{\hglue -1pt$\displaystyle \mathop{\bullet}^{_7}$} \hglue -4pt\hrulefill\lower 2pt
\hbox{\hglue -1pt$\displaystyle \mathop{\bullet}^{_8}$}
\hglue -4pt\hrulefill\
\lower 2pt\hbox{\hglue -4pt$\displaystyle\mathop{\bullet}^{_0}$}
\hglue -4pt\hrulefill\lower 2pt\hbox{\hglue -1pt$\displaystyle{\bullet}^{_{-1}}$}
}
\end{center}
The corresponding generalized Cartan matrix $A^J$ is  hyperbolic of type $HF_{4}^{(1)}$ :
\begin{center}{
 \hskip 3cm\hbox to 4cm{ \offinterlineskip\lower 2pt
\hbox{\hglue -1pt $\displaystyle\mathop{\bullet}_{^1}$}
\hglue -5pt\hrulefill\lower 2pt
\hbox{ \hglue -4pt\hbox{$\displaystyle\mathop{\bullet}_{^6}$}\hglue -2pt
\vbox{  {
\hrule height 0,3pt width 0,7cm}\vskip 2,5pt{\hrule height 0,3pt width 0,7cm}
\vskip 1pt}\hglue -12pt $<$
\hglue -0.5pt \hbox{$\displaystyle\mathop{\bullet}_{^7}$} }\hglue -5pt
\hrulefill \lower 2pt \hbox{\hglue -1pt $\displaystyle\mathop{\bullet}_{^8}$}\hglue -1pt\hrulefill\lower 2pt\hbox{$\displaystyle\mathop{\bullet}_{^0}$}\hglue -5pt\hrulefill\lower 2pt
\hbox{\hglue -3pt$\displaystyle\mathop{\bullet}_{^{-1}}$} }
}
\end{center}
2) $J=\{1,2,3,4,5,6\}$. 
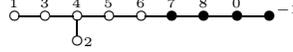
\begin{center}
\hskip 3cm \hbox to 4cm{\lower 2pt
\hbox{\hglue -1pt$\displaystyle\mathop{\circ}^{_1}$}
\hglue -4.5pt\hrulefill\lower 2pt\hbox{\hglue -1pt$\displaystyle
\mathop{\circ}^{_3}$} \hglue -4pt\hrulefill\lower 2pt
\hbox{\hglue -1pt
\vtop{ \offinterlineskip \hbox{$\displaystyle\mathop{\circ}^{_4}$}
\vskip -1pt
\hbox to 5pt{\hfill\vrule height 6pt width 0,3pt\hfill}
\hbox{${\circ} _{^2}$\hfill}
} }
\hglue -12pt\hrulefill\lower 2pt\hbox{\hglue -1pt$\displaystyle
\mathop{\circ}^{_5}$} \hglue -4pt\hrulefill\lower 2pt
\hbox{\hglue -1pt$\displaystyle \mathop{\circ}^{_6}$}\hglue -1pt\hrulefill\lower 2pt
\hbox{\hglue -1pt$\displaystyle \mathop{\bullet}^{_7}$} \hglue -4pt\hrulefill\lower 2pt
\hbox{\hglue -1pt$\displaystyle \mathop{\bullet}^{_8}$}
\hglue -4pt\hrulefill\
\lower 2pt\hbox{\hglue -4pt$\displaystyle\mathop{\bullet}^{_0}$}
\hglue -4pt\hrulefill\lower 2pt\hbox{\hglue -1pt$\displaystyle{\bullet}^{_{-1}}$}
}
\end{center}
The corresponding generalized Cartan matrix $A^J$ is  hyperbolic of type $HG_{2}^{(1)}$ :
 \begin{center}
\hskip 3cm \hbox to 4cm{
\hbox{$\displaystyle\mathop{\bullet}_{^7}$}\hglue 0pt
\hbox{\offinterlineskip \hglue -2.5pt\vbox{{\hrule height 0,3pt width 1,1cm}\vskip 1,5pt{\hrule height 0,3pt width 1,1cm}\vskip 1,5pt {\hrule height 0,3pt width 1,1cm}\vskip 0.76pt}\hglue -18pt \hbox{$<$ \hglue 4.5pt 
\hbox{$\displaystyle\mathop{\bullet}_{^8}$} } \hglue -9pt\vbox{ {\hrule height 0,3pt width 1,6cm}\vskip 2pt}\hglue -22pt\hbox{$\displaystyle\mathop{\bullet}_{^0}$}\hglue -3pt\vbox{ {\hrule height 0,3pt width 0.9cm}\vskip 2pt}\hglue -3pt \hbox{$\displaystyle\mathop{\bullet}_{^{-1}}$}\hglue -1pt }}
\end{center}

 Note that the first example corresponds to an almost split real form of the Kac-Moody Lie algebra $\mathfrak{g}(A)$ and $A^J$ is the generalized Cartan matrix associated to the corresponding (reduced) restricted root system (see \cite{bm}) whereas the second example does not correspond  to an almost split real form of $\mathfrak{g}(A)$.
 \end{exa}
 \begin{lem}\label{gki}
For $k\in I'$, set $\mathfrak{s}(k)=\Complex E_k\oplus\Complex H_k\oplus\Complex F_k$.  
   Then, the Kac-Moody algebra  $\mathfrak{g}$ is an integrable $\mathfrak{s}(k)-$module via the adjoint representation of $\mathfrak{s}(k)$ on $\mathfrak{g}$.
\end{lem}
\begin{proof} Note that $\mathfrak{s}(k)$ is isomorphic to $\mathfrak{sl}_2(\Complex)$ with standard basis $(E_k, H_k, F_k)$. It is clear that  ad$(H_k)$ is diagonalizable on $\mathfrak{g}$ and $E_k=\sum_{\alpha} e_{\alpha}\in d_{k,1}$, where $\alpha$  runs over the set $\Delta_{k,1}=\{\alpha\in \Delta(I_k); \langle\alpha,H_k\rangle=2\}$,  $e_{\alpha}\in \mathfrak{g}_{\alpha}$ for $\alpha\in \Delta(I_k)$, and $d_{k,1}:=\{X\in \mathfrak{g}(I_k) \, ;\; [H_k, X]=2X\}$. Since  $\Delta_{k,1}\subset \Delta^{re}$,  ad($e_{\alpha}$) is locally nilpotent  for $\alpha\in\Delta_{k,1}$. As $d_{k,1}$ is commutative (see Remark \ref{eqc})  we deduce that ad$(E_k)$ is locally nilpotent on $\mathfrak{g}$. The same argument shows that  ad$(F_k)$ is also locally nilpotent. Hence, the Kac-Moody algebra $\mathfrak{g}$ is an integrable  $\mathfrak{s}(k)-$module.
\end{proof}
\begin{prop} \label{gJ} Let
$\mathfrak{g}^J$ be the subalgebra of $\mathfrak{g}$ generated by $\mathfrak{h}^J$ and $E_k, F_k$, $k\in I'$. Then $\mathfrak{g}^J$ is the Kac-Moody Lie algebra associated to the realization  $(\mathfrak{h}^J, \Pi^J, \Pi^{J\vee})$  of the generalized Cartan matrix $A^J$.
\end{prop}
\begin{proof}
It is  not difficult to check that the  following relations hold  in the Lie subalgebra $\mathfrak{g}^J$ : 

$
\begin{array}{lll}
[\mathfrak{h}^J, \mathfrak{h}^J]=0, & [E_k,F_l]=\delta_{k,l}H_k & (k,l \in I');\\
% \, & \,  & \,  \\
 
[h, E_k] = \langle\alpha_k', h\rangle E_k, & [h, F_k]= -\langle\alpha_k', h\rangle F_k & (h \in \mathfrak{h}^J,  k\in I').\\
 &  &
 \end{array}
$

\noindent We have to prove the Serre's relations : 
$$
\begin{array}{lll}
 (\text{ad}E_k)^{1-a'_{k,l}}(E_l)= 0, & (\text{ad}F_k)^{1-a'_{k,l}}(F_l)= 0  & (k\not= l\in I'). \\
 &  &
\end{array}
$$
\noindent
For $k\in I'$, let $\mathfrak{s}(k)=\Complex F_k\oplus\Complex H_k\oplus\Complex E_k$ be the Lie  subalgebra of $\mathfrak{g}$ isomorphic to $\mathfrak{sl}_2(\Complex)$.  Let $l\not= k\in I'$; note that 
$[H_k, F_l]=-a'_{k,l}F_l$ and  $[E_k,F_l]=0$, which means that $F_l$ is a primitive weight vector for $\mathfrak{s}(k)$.  As $\mathfrak{g}$ is an integrable $\mathfrak{s}(k)-$module (see Lemma \ref{gki}) the primitive weight vector $F_l$  is contained in a finite dimensional $\mathfrak{s}(k)-$submodule (see \cite{vk}; 3.6). The relation  $(\text{ad}F_k)^{1-a'_{k,l}}(F_l)= 0$ follows from the representation theory of  $\mathfrak{sl}_2(\Complex)$ (see \cite{vk}; 3.2).  By similar arguments we prove that $(\text{ad}E_k)^{1-a'_{k,l}}(E_l)= 0$. 

\par Now $\mathfrak{g}^J$ is a quotient of the Kac-Moody algebra associated to $A^J$ and $(\mathfrak{h}^J, \Pi^J, \Pi^{J\vee})$. By \cite[1.7]{vk} it is equal to it.%\marginpar{modif}
\end{proof}
\begin{defn} The Kac-Moody Lie algebra $\mathfrak{g}^J$ is called the $C-$admissible  algebra associated to the  $C-$admissible pair 
$(I,J)$.
\end{defn}

\begin{prop}\label{finmult}  The Kac-Moody algebra  $\mathfrak{g}$ is an integrable $\mathfrak{g}^J-$module with finite multiplicities.
\end{prop}
\begin{proof}  The  $\mathfrak{g}^J-$module  $\mathfrak{g}$ is clearly ad$(\mathfrak{h}^J)-$diagonalizable and   ad$(E_k)$, ad$(F_k)$  are locally nilpotent on $\mathfrak{g}$ for $k\in I'$ (see Lemma \ref{gki}).  Hence, $\mathfrak{g}$ is an integrable $\mathfrak{g}^J-$module.
 For $\alpha\in \Delta$, let $\alpha'=\alpha_{\vert\mathfrak{h}^J}$ be the restriction of $\alpha$ to $\mathfrak{h}^J$.  Set $\Delta'=\{\alpha' ; \, \alpha\in\Delta\}\setminus\{0\}$. Then
the set of weights, for the $\mathfrak{g}^J-$module  $\mathfrak{g}$, is exactly $\Delta'\cup\{0\}$. Note that for $\alpha\in \Delta$,  $\alpha'=0$ if and only if $\alpha\in\Delta(J)$. 
In particular, the weight space $V_0 =\mathfrak{h}\oplus(\displaystyle\mathop{\oplus}_{\alpha\in\Delta(J)}\mathfrak{g}_{\alpha})$  corresponding to the null weight  is finite dimensional. Let $\alpha=\displaystyle\sum_{i\in I}n_i\alpha_i\in\Delta$ such that $\alpha'\not=0$. 
We will see that the corresponding weight space $V_{\alpha'}$ is finite dimensional. Note that $V_{\alpha'}= \displaystyle\mathop{\oplus}_{\beta'=\alpha'}\mathfrak{g}_{\beta}$. Let  $\beta=\displaystyle\sum_{i\in I}m_i\alpha_i\in\Delta$ such that $\beta'=\alpha'=\displaystyle\sum_{k\in I'}n_k\alpha_k'$, then $m_k=n_k$, $\forall k\in I'$, since $\Pi^J=\{\alpha'_k, k\in I'\}$ is free in $(\mathfrak{h}^J)^*$. In particular, $\beta$ and $\alpha$ are of the same sign, and we may assume $\alpha\in\Delta^+$. 
Let $ht_J(\beta)=\displaystyle\sum_{j\in J}m_j$ be the height of $\beta$ relative to $J$, and let $W_J$ be the finite subgroup of $W$ generated by $r_j$, $j\in J$.  Since $ W_J$ fixes pointwise $\mathfrak{h}^J$, we deduce that $\gamma'=\beta'$, $\forall \gamma\in W_J\beta$, and so we may assume that $ht_J(\beta)$ is minimal among the roots in $ W_J\beta$. From the inequality $ht_J(\beta)\leq ht_J(r_j(\beta))$, $\forall j\in J$, we get  $\langle\beta,\alpha\check{_j}\rangle\leq 0$, $\forall j\in J$. Let $\rho\check{_J}$ be the half sum of positive coroots of $\Delta(J)$. It is known that $\langle\alpha_j,\rho\check{_J}\rangle=1$, $\forall j\in J$. It follows that $0\geq \langle\beta,\rho\check{_J}\rangle= \displaystyle\sum_{j\in J}m_j+\sum_{k\in I'}n_k\langle\alpha_k,\rho\check{_J}\rangle$,  so finally, we obtain :  $ht_J(\beta)\leq\displaystyle\sum_{k\in I'}-n_k\langle\alpha_k,\rho\check{_j}\rangle$.  Hence, there is just a finite number of possibilities for $\beta$, and then $\alpha'$ is of finite multiplicity. 
\end{proof}

\begin{thm}\label{jgrad} Let $\Delta^J$ be the root system of the pair $(\mathfrak{g}^J,\mathfrak{h}^J)$, then the Kac-Moody Lie algebra $\mathfrak{g}$ is finitely $\Delta^J-$graded, with grading subalgebra $\mathfrak{g}^J$. 
\end{thm}

\begin{proof} Let $\Delta'=\{\alpha', \alpha\in\Delta\}\setminus\{0\}$ be the set of nonnull weights of the $\mathfrak{g}^J-$module  $\mathfrak{g}$ relative to $\mathfrak{h}^J$.  Let $\Delta'_+=\{\alpha'\in\Delta', \alpha\in\Delta^+\}$ and $\Delta^J_+$ the set of positive roots of $\Delta^J$ relative to the root basis $\Pi^J$. We have to prove that $\Delta'=\Delta^J$ or equivalently $\Delta'_+=\Delta^J_+$. 
Let $Q^J=\Z \Pi^J$ be the root lattice of $\Delta^J$ and $Q^J_+=\Zplus \Pi^J$. 
It is known that the positive root system $\Delta^J_+$ is uniquely defined by the following  properties (see \cite{vk}, Ex. 5.4) :
\\
(i) $\Pi^J\subset\Delta^J_+\subset Q^J_+$, $2\alpha_i'\notin\Delta^J_+$, $\forall i\in I'$;
\\
(ii) if $\alpha'\in\Delta^J_+$, $\alpha'\not=\alpha'_i$, then the set $\{\alpha'+k\alpha'_i; k\in\Z\}\cap \Delta^J_+$ is a string\goodbreak $\{\alpha'-p\alpha'_i, ...., \alpha'+q\alpha'_i\}$, where $p,q\in\Zplus$ and $p-q=\langle \alpha',H_i\rangle$;
\\
(iii) if $\alpha'\in\Delta^J_+$, then supp$(\alpha')$ is connected.\\
We will see that $\Delta'_+$ satisfies these three properties and hence $\Delta'_+=\Delta^J_+$. Clearly $\Pi^J\subset\Delta'_+\subset Q^J_+$. 
For $\alpha\in \Delta$ and $k\in I'$, the condition $\alpha'\in\Nat\alpha_k$ implies $\alpha\in \Delta(I_k)^+$. As $(I_k,J_k)$ is $C-$admissible for $k\in I'$, the highest root of $\Delta(I_k)^+$ has coefficient 1 on the root $\alpha_k$ (cf. Remark \ref{remadm}).
 It follows that $2\alpha'_k\notin  \Delta'_+$ and (i) is satisfied.  By Proposition \ref{finmult}, $\mathfrak{g}$ is an integrable $\mathfrak{g}^J-$module with finite multiplicities. Hence,  the propriety (ii) follows from [\cite{vk}; prop.3.6]. Let  $\alpha\in\Delta_+$, then supp$(\alpha)$ is connected and supp$(\alpha')\subset$ supp$(\alpha)$.  Let   $k,l\in$ supp$(\alpha')$; if $k,l$ are $J-$ connected in supp$(\alpha)$ relative to the generalized Cartan matrix $A$ (cf. \ref{jc}), then by lemma \ref{ljc}, $k,l$    are linked in $I'$ relative to the generalized Cartan matrix $A^J$. Hence, the connectedness of   supp$(\alpha')$, relative to $A^J$,  follows from that of  supp$(\alpha)$ relative to $A$ (see Remark \ref{rjc}) and (iii) is satisfied. 
\end{proof}

\begin{rem} Note that the definition of $C-$admissible pair can be extended to decomposable Kac-Moody Lie algebras : 
thus if $I^1, I^2, ...., I^m$ are the connected components of $I$ and $J^k=J\cap I^k$,   $k=1,2, ...., m$, then $(I, J)$ is $C-$admissible if and only if $(I^k, J^k)$ is for all  $k=1,2, ...., m$. 
In particular, the corresponding $C-$admissible algebra is $\mathfrak g^J=\displaystyle\mathop{\oplus}_{k=1}^m\mathfrak{g}(I^k)^{J^k}$, where $\mathfrak{g}(I^k)^{J^k}$ is the $C-$admissible subalgebra of $\mathfrak{g}(I^k)$ corresponding to $(I^k, J^k)$,  $k=1,2, ...., m$.

\end{rem}

\section{General  gradations.}\label{GenGrad} $\,$\\
Let $\mathfrak{m}$ be an indecomposable Kac-Moody subalgebra of $\mathfrak{g}$ and let $\mathfrak{a}$ be a Cartan subalgebra of $\mathfrak{m}$. Put $\Sigma=\Delta(\mathfrak{m}, \mathfrak{a})$ the corresponding root system. We assume that $\mathfrak{a}\subset\mathfrak{h}$ and that $\mathfrak{g}$ is finitely $\Sigma-$graded with $\mathfrak{m}$ as grading subalgebra.
Thus $\mathfrak{g}=\displaystyle\sum_{{\gamma}\in\Sigma \cup\{0\}}V_{{\gamma}}$, with $V_{{\gamma}}=\{x\in \mathfrak{g}\, ; \;
[a, x]=\langle \gamma, a\rangle x, \; \forall a \in{\mathfrak{a}}\}$ is finite dimensional for all  ${\gamma}\in\Sigma \cup\{0\}$.
For $\alpha\in\Delta$, denote by $\rho_a(\alpha)$ the restriction of $\alpha$ to $\mathfrak{a}$. As $\mathfrak{g}$ is  $\Sigma-$graded, one has $\rho_a(\Delta\cup\{0\})= \Sigma\cup\{0\}$.
\begin{lem} 
$\,$\\
1) Let $\mathfrak{c}$ be the center of $\mathfrak{g}$ and denote by 
$\mathfrak{c}_a$  the center of $\mathfrak{m}$. Then $\mathfrak{c}_a= \mathfrak{c}\cap\mathfrak{a}$. In particular, if $\mathfrak{g}$ is perfect, then the grading subalgebra $\mathfrak{m}$ is also perfect.\\
2) Suppose that $\Delta^{im}\not=\emptyset$, then   $\rho_a( \Delta^{im})\subset  \Sigma^{im}$.
\end{lem}
\begin{proof}$\,$\\
1) It is clear that $\mathfrak{c}\cap\mathfrak{a}\subset \mathfrak{c}_a $. Since $\mathfrak{g}$ is $\Sigma-$ graded, we deduce that $\mathfrak{c}_a$ is contained in the center $\mathfrak{c}$ of $\mathfrak{g}$, hence $\mathfrak{c}_a\subset\mathfrak{c}\cap\mathfrak{a}$.
 If $\mathfrak{g}$ is perfect, then $\mathfrak{g}=\mathfrak{g}'$, $\mathfrak{h}=\mathfrak{h}'$, $\mathfrak{c}=\{0\}$; so  $\mathfrak{c}_a=\{0\}$, $\mathfrak{a}=\mathfrak{a}'$ and $\mathfrak{m}=\mathfrak{m}'$.
 \\
2) If $\alpha\in\Delta^{im}$, then $\Nat\alpha\subset\Delta$.  Since $V_0$ is finite dimensional, $\rho_a(\alpha)\not=0$  and $\Nat\rho_a(\alpha)\subset\Sigma$, hence $\rho_a(\alpha)\in\Sigma^{im}$.
\end{proof}
In the following, we will show that the Kac-Moody Lie algebra $\mathfrak{g}$ and the grading subalgebra $\mathfrak{m}$ are of the same type.
\begin{lem}\label{ind} The Kac-Moody Lie algebra $\mathfrak{g}$ is of indefinite type if and only if   $\Delta^{im}$ generates the dual space $(\mathfrak{h}/\mathfrak{c})^*$ of $\mathfrak{h}/\mathfrak{c}$. 
\end{lem}
\begin{proof}  Note that the root basis $\Pi=\{\alpha_i, i\in I\}$ induces a basis of the vector space $(\mathfrak{h}/\mathfrak{c})^*$. In particular, dim$(\mathfrak{h}/\mathfrak{c})^*\geq 2$ when $\Delta^{im}$ is nonempty. 
Suppose now that $\mathfrak{g}$ is of indefinite type. Let  $\alpha\in \Delta_+^{sim}$ be a positive strictly imaginary root satisfying $\langle\alpha, \alpha\check{_i}\rangle <0$, $\forall i\in I$; then, $r_i(\alpha)= \alpha- \langle\alpha, \alpha\check{_i}\rangle\alpha_i\in\Delta_+^{im}$ for all $i\in I$. In particular, the vector subspace $\langle\Delta^{im}\rangle$  spanned by $\Delta^{im}$ contains $\Pi$ and hence is equal to $(\mathfrak{h}/\mathfrak{c})^*$.  Conversely, if $\Delta^{im}$ generates $(\mathfrak{h}/\mathfrak{c})^*$, then $\Delta^{im}$ 
is nonempty and contains at least two linearly independent imaginary roots;  hence $\Delta$ can not be of finite or affine type.
\end{proof}

\begin{prop}\label{st} $\,$\\
1) The Kac-Moody Lie Algebra $\mathfrak{g}$ and the grading subalgebra $\mathfrak{m}$ are of the same type.\\
2) Suppose $\mathfrak{g}$ of indefinite type and Lorentzian, then $\mathfrak{m}$ is also Lorentzian.
\end{prop}
\begin{proof}$\,$\\
1) If $\mathfrak{g}$ is of finite type, then $\Delta$ is finite and hence $\Sigma=\rho_a(\Delta)\setminus\{0\}$ is finite. \\
If $\mathfrak{g}$ is  is affine, let $\delta$ be the lowest positive imaginary root. One can choose a root basis  $\Pi_a=\{\gamma_i, i\in \bar I\}$ of $\Sigma$ so that $\bar\delta:=\rho_a(\delta)$ is a positive imaginary root. Note that $\mathfrak{a}':=\mathfrak{a}\cap\mathfrak{m}'\subset \mathfrak{h}'$; in particular $\bar\delta(\mathfrak{a}') =\{0\}$ and $\langle\bar\delta, \gamma\check{_i}\rangle =0$, $\forall i\in \bar I$.  It follows that $\mathfrak{m}$ is  is affine (see [\cite{vk}; Prop. 4.3]).\\
Suppose now that $\mathfrak{g}$ is  is of indefinite type. Thanks to Lemma \ref{ind}, it suffices to prove that $\Sigma^{im}$ generates $(\mathfrak{a}/\mathfrak{c}_a)^*$, where $\mathfrak{c}_a=\mathfrak{c}\cap\mathfrak{a}$ is the center of $\mathfrak{m}$.  The natural homomorphism of vector spaces  $\pi : \mathfrak{a}\rightarrow\mathfrak{h}/\mathfrak{c}$  induces a monomorphism $\bar\pi :  \mathfrak{a}/\mathfrak{c}_a\rightarrow\mathfrak{h}/\mathfrak{c}$. By duality,  the homomorphism ${\bar{\pi}}^* : (\mathfrak{h}/\mathfrak{c})^*\rightarrow (\mathfrak{a}/\mathfrak{c}_a)^*$ is surjective and ${\bar{\pi}}^*(\Delta^{im})\subset \Sigma^{im}$  generates $(\mathfrak{a}/\mathfrak{c}_a)^*$.\\
2) Suppose that  $\mathfrak{g}$ is Lorentzian and let $(.\, , .)$ be an  invariant nondegenerate bilinear form on $\mathfrak{g}$. Then, the restriction
of $(.\, , .)$ to $\mathfrak{h}_{\Real}$ has signature $(++ ....+,-)$ and any maximal totally isotropic subspace of $\mathfrak{h}_{\Real}$ relatively to  $(.\, , .)$ is one dimensional.   
Let $\mathfrak{a}_{\Real}:=\mathfrak{a}\cap\mathfrak{h}_{\Real}$ and let $(.\, , .)_a$ be the restriction of $(.\, , .)$ to $\mathfrak{m}$.  
As $\mathfrak{m}$ is of indefinite type, dim $(\mathfrak{a})\geq 2$ and  the restriction 
of $(.\, , .)_a$ to $\mathfrak{a}_{\Real}$   is nonnull. 
 It follows that the orthogonal subspace $\mathfrak{m}^{\perp}$ of $\mathfrak{m}$  relatively to $(.\, , .)_a$ is a proper ideal of $\mathfrak{m}$. Since $\mathfrak{m}$ is perfect (because $\mathfrak{g}$ is) we deduce that $\mathfrak{m}^{\perp}=\{0\}$ (cf. \cite[1.7]{vk}) 
  and the invariant bilinear form 
$(.\, , .)_a$ is nondegenerate. It follows that  $\mathfrak{m}$ is symmetrizable and  the restriction of $(.\, , .)_a$ to $\mathfrak{a}$ is  nondegenerate. As $\mathfrak{m}$ is of indefinite type, the restriction of $(.\, , .)_a$ to $\mathfrak{a}_{\Real}$ can  not be positive definite.  Hence, the bilinear form  $(.\, , .)_a$ has signature $(++ ....+,-)$ on $\mathfrak{a}_{\Real}$ and then the grading subalgebra $\mathfrak{m}$ is Lorentzian. 
\end{proof}
\begin{defn}
Let $\Pi_a$ be a root basis of $\Sigma$ and let $\Sigma^+$ be the corresponding set of positive roots. The root basis is said to be adapted to the root basis $\Pi$ of $\Delta$ if $\rho_a(\Delta^+)\subset \Sigma^+ \cup\{0\}$. 
\end{defn}
We will see that adapted  root bases always exist.

\begin{defn} (\cite{bp}; 5.2.6) Suppose that $\Delta^{im}\not=\emptyset$. 
Let $\alpha,\beta\in\Delta^{im}$. \\
(i) The imaginary roots  $\alpha$ and $\beta$ are said to be linked if $\Nat\alpha+\Nat\beta\subset\Delta$ or $\beta\in \Rat^+\alpha$.
\\
(ii) The imaginary roots  $\alpha$ and $\beta$ are said to be linkable  if there exists a finite family of imaginary roots $(\beta_i)_{0\leq i\leq n+1}$ such that $\beta_0=\alpha$, $\beta_{n+1}=\beta$ and $\beta_i$ and $\beta_{i+1}$ are linked for all $i=0,1, ...., n$.
\end{defn}
\begin{prop}(\cite{bp}; Prop. 5.2.7)\label{lk} Suppose that $\Delta^{im}\not=\emptyset$. To be linkable is an equivalence relation on  $\Delta^{im}$ and, if $A$ is indecomposable, there exist exactly two equivalence classes :  $\Delta_-^{im}$ and  $\Delta_+^{im}$. 
\end{prop}
\begin{lem}\label{pab} Suppose that $\Delta^{im}\not=\emptyset$, then    there exists a root basis of $\Sigma$ such that 
$\rho_a( \Delta_+^{im})\subset  \Sigma_+^{im}$.
\end{lem}
\begin{proof}  Let $\alpha,\beta\in \Delta_+^{im}$, then, by  Proposition \ref{lk},  $\alpha$ and $\beta$ are linkable and so are $\rho_a(\alpha)$ and $\rho_a(\beta)$. Since $\Sigma$ is assumed to be indecomposable, $\rho_a(\alpha)$ and $\rho_a(\beta)$  are of the same sign. One can choose a root basis of $\Sigma$ such that  $\rho_a(\alpha)$ and $\rho_a(\beta)$  are positive, and then we have $\rho_a( \Delta_+^{im}) \subset  \Sigma_+^{im}$.  
\end{proof}
\begin{cor} Let $\Pi_a$ be a root basis of $\Sigma$ such that $\rho_a( \Delta_+^{im}) \subset  \Sigma_+^{im}$ and let  $X_a$ be the corresponding positive Tits cone. Then we have  $\bar{X}_a\subset\bar{X}\cap\mathfrak{a}$.
\end{cor}
\begin{proof} As $\Delta^{im}\not=\emptyset$, one has $\bar X=\{h\in\mathfrak{h}_{\Real}; \langle \alpha, h\rangle\geq 0, \forall \alpha\in \Delta_+^{im}\}$ (see Remark \ref{tc}). The corollary follows from   Lemma \ref{pab}.
\end{proof}
\begin{lem}\label{int} Suppose that $\Delta^{im}\not=\emptyset$. Let $p\in\bar X$ such that $ \langle \alpha, p\rangle\in\Z$, $\forall \alpha\in\Delta$,  and $ \langle \beta, p\rangle>0$,    $\forall \beta\in \Delta_+^{im}$. Then $p\in\buildrel\circ\over X$.
\end{lem}
\begin{proof} $\,$\\The result is clear when $\Delta$ is of affine type since  $\buildrel\circ\over{X}=\buildrel\circ\over{\bar X} =\{h\in\mathfrak{h}_{\Real}; \langle\delta, h\rangle >0\}$. Suppose now that   $\Delta$ is of indefinite type.  If one looks to the proof of Proposition 5.8.c) in \cite{vk}, one can show that an element $p\in \bar X$ satisfying the conditions of the lemma lies in $X$. As $\Delta_+^{im}$ is $W-$invariant, we may assume that $p$ lies in the fundamental chamber $C$. Hence there exists a subset $J$ of $I$ such that $\{\alpha\in\Delta; \, \langle\alpha,p\rangle=0\}=\Delta_J=\Delta\cap\sum_{j\in J}\Z\alpha_j$. Since $\Delta_J\cap\Delta^{im}=\emptyset$, the root subsystem $\Delta_J$ is of finite type and $p$ lies in the finite type facet  of type $J$.
\end{proof}

\begin{thm}\label{adap} There exists a root basis $\Pi_a$ of $\Sigma$ which is adapted to the root basis $\Pi$ of $\Delta$. 
Moreover, there exists a finite type subset $J$ of $I$ such that $\Delta_J= \{\alpha\in\Delta; \; \rho_a(\alpha)=0\}$.
\end{thm}
\textbf{N.B.} This is part 1) of Theorem 2. %\ref{thmB}. \marginpar{modif}
\begin{proof} Let $\Pi_a=\{\gamma_i, i\in \bar I\}$ be a root basis of $\Sigma$ such that  $\rho_a( \Delta_+^{im})\subset  \Sigma_+^{im}$, where $\bar I$ is just a set indexing the basis elements. Let $p\in\mathfrak{a}$ such that $\langle\gamma_i, p\rangle =1$, $\forall i\in \bar I$ and let ${P}= \{\alpha\in\Delta; \langle\alpha, p\rangle\geq 0\}$. If $\Delta$ is finite, then ${P}$ is clearly a parabolic subsystem of $\Delta$ and the result is trivial. Suppose now that $\Delta^{im}\not=\emptyset$; then $p$ satisfies the conditions of the Lemma \ref{int} and we may assume that $p$ lies in the facet of type $J$ for some subset $J$ of finite type in $I$. %\marginpar{modif}
 In which case ${P}= \Delta_J\cup\Delta^+$ is the standard parabolic subsystem of finite type $J$. Note that, for $\gamma\in\Sigma^+$, one has $\langle\gamma, p\rangle =ht_a(\gamma)$ the height of $\gamma$ with respect to $\Pi_a$. It follows that $\{\alpha\in\Delta; \rho_a(\alpha)=0\}=\Delta_J$, in particular, $\rho_a(\Delta^+)=\rho_a({P}) \subset \Sigma^+ \cup\{0\}$. Hence, the root basis $\Pi_a$ is adapted to $\Pi$.
\end{proof}

From now on,  we fix a root basis  $\Pi_a=\{\gamma_s, s\in \bar I\}$, for the grading root system $\Sigma$, which is adapted to the root basis $\Pi=\{\alpha_i, i\in I\}$ of $\Delta$ (see Theorem \ref{adap}). As before, let $J:=\{j\in I\, ; \, \rho_a(\alpha_j)=0\}$ and $I':= I\setminus J$. For $k\in I'$, we denote, as above, by $I_k$ the connected component of $J\cup\{k\}$ containing $k$, and $J_k:=J\cap I_k$.

\begin{prop}\label{rsr} $\,$\\
1) Let $s\in \bar I$, then  there exists $k_s\in I'$ such that $\rho_a(\alpha_{k_s})=\gamma_s$ and any preimage  $\alpha	\in \Delta$ of $\gamma_s$ is equal to $\alpha_k$ modulo $\sum_{j\in J_k}\,\Z\alpha_j$ 
for some $k\in I'$ satisfying $\rho_a(\alpha_{k})=\gamma_s$.  \\
2) Let $k\in  I'$  such that $\rho_a(\alpha_{k})$  is a real root of $\Sigma$. Then $\rho_a(\alpha_{k})\in\Pi_a$ is a simple root. 
\end{prop}
\begin{proof} This result was proved by J. Nervi for affine algebras (see \cite{nerv2}, Prop.2.3.10 and the proof of Prop. 2.3.12). The arguments used there are available for  general Kac-Moody algebras.
\end{proof}
We introduce the following notations :   $$I'_{re}:=\{i\in I'\, ; \, \rho_a(\alpha_i)\in\Pi_a\}\;  ;  \; \;I'_{im}:=I'\setminus I'_{re},$$  $$I_{re}=\displaystyle\mathop{\cup}_{k\in I'_{re}} I_k \; ; \; \; J_{re} =I_{re}\cap J=\displaystyle\mathop{\cup}_{k\in I'_{re}} J_k \;  ; \;\; J^{\circ}=J\setminus J_{re}$$ 
 $$\Gamma_s:=\{i\in I'\, ; \, \rho_a(\alpha_i)=\gamma_s\}\, ,\forall s\in \bar I.$$ 
 Note that $J^{\circ}$ is not connected to $I_{re}$.
\begin{rem}\label{gap} $\,$\\
1) In view  of  Proposition \ref{rsr}, assertion 2),   one has $\rho_a(\alpha_k)\in\Sigma_{im}^+$,  $\forall k\in I'_{im}$.\\
2) $I=I_{re}\cup I'_{im}\cup J^{\circ}$ is a disjoint union. \\
3) If $I'_{im}=\emptyset$, then  $I=I_{re}\cup J^{\circ}$. Since $I$ is connected (and $I_{re}$ is not connected to $J^{\circ}$) we deduce that  $J^{\circ}=\emptyset$,   $I=I_{re}$ and $I'_{re}=I'=I\setminus J$.
 \end{rem}
 
\begin{prop}\label{ire}  $\,$\\
1) Let $k\in I'_{re}$, then $I_k$ is of finite type.\\
2) Let $s\in \bar I$. If $|\Gamma_s|\geq 2$ and $k\not=l\in\Gamma_s$, then $I_k\cup I_l$ is not connected: $\mathfrak{g}(I_k)$ and $ \mathfrak{g}(I_l)$ commute and are orthogonal.
\\
3) For all  $k\in  I'_{re}$, $(I_k,J_k)$ is an irreducible $C-$admissible pair. \\
4) The  derived subalgebra  $\mathfrak{m}'$ of the grading algebra $\mathfrak{m}$ is contained in $\mathfrak{g}'(I_{re})$ (as defined in proposition \ref{realj}).
\end{prop}

\begin{proof} $\,$ \\
1) Suppose that there exists  $k\in I'_{re}$ such that  $I_k$ is not of finite type; then there exists an imaginary root $\beta_k$ whose support is the whole $I_k$. Hence, there exists a positive integer $m_k\in\Nat$ such that $\rho_a(\beta_k)=m_k\rho(\alpha_k)$ is an imaginary root of $\Sigma$. It follows that $\rho_a(\alpha_k)$ is  an  imaginary root and this contradicts the fact that $k\in I'_{re}$.\\
2) Let $s\in \bar I$ such that $|\Gamma_s|\geq 2$  and let $k\not=l\in\Gamma_s$. Since $V_{n\gamma_s}=\{0\}$ for all integer $n\geq 2$, the same argument used in 1) shows that $I_k\cup I_l$ is not connected, and  $I_k$ and $I_l$ are its  two connected components. In particular, $[\mathfrak{g}(I_k), \mathfrak{g}(I_l)]=\{0\}$ and $(\mathfrak{g}(I_k), \mathfrak{g}(I_l))=\{0\}$.\\ 
3) Let $k\in  I'_{re}$ and let  $s\in \bar I$ such that $\rho_a(\alpha_k)=\gamma_s$.  Let $(\bar X_s, \; \bar H_s=\gamma\check{_s}, \; \bar Y_s)$ be an $\mathfrak{sl}_2-$triple in $\mathfrak{m}$ corresponding to the simple root $\gamma_s$. Let  $V_{\gamma_s}$ be the weight space of $\mathfrak{g}$ corresponding to $\gamma_s$. In view of  Proposition \ref{rsr}, assertion 1),  one has : 
\begin{equation}\label{wsd} 
V_{\gamma_s}= \displaystyle\mathop{\oplus}_{l\in\Gamma_s}V_{\gamma_s}\cap\mathfrak{g}(I_l). 
\end{equation}
Hence, one can  write : 
\begin{equation}\label{xbar} 
\bar X_s=\sum_{l\in\Gamma_s}E_l\, ; \quad\bar Y_s=\sum_{l\in\Gamma_s}F_l\, ,
\end{equation}
with $E_l\in V_{\gamma_s}\cap\mathfrak{g}(I_l)$ and $F_l\in V_{-\gamma_s}\cap\mathfrak{g}(I_l)$. It follows from assertion 1) that \begin{equation}\label{hbar} 
\bar H_s=\gamma\check{_s}=[\bar X_s,\bar Y_s]=\sum_{l\in\Gamma_s}[E_l, F_l]= \sum_{l\in\Gamma_s}H_l ,
\end{equation}
where $H_l:=[E_l, F_l]\in\mathfrak{h}_{I_l}$, $\forall l\in\Gamma_s$. Then one has, for $ k\in\Gamma_s$,  $$2=\langle\gamma_s,\gamma\check{_s}\rangle=\langle\alpha_k,\gamma\check{_s}\rangle=\sum_{l\in\Gamma_s}\langle\alpha_k,H_l\rangle=\langle\alpha_k,H_k\rangle\, , $$   and for $j\in J_k$,  $$0=\langle\alpha_j,\gamma\check{_s}\rangle=\sum_{l\in\Gamma_s}\langle\alpha_j,H_l\rangle=\langle\alpha_j,H_k\rangle.$$ In particular, $H_k$ is the unique semi-simple element of $\mathfrak{h}_{I_k}$ satisfying : 
\begin{equation}\label{hk}
\langle\alpha_i,H_k\rangle =2\delta_{i,k}, \forall i\in I_k. 
\end{equation} 
Hence,  $(E_k, H_k, F_k)$ is an $\mathfrak{sl}_2-$triple in the simple Lie algebra  $\mathfrak{g}(I_k)$ and since $V_{2\gamma_s}=\{0\}$, $(I_k,J_k)$ is an irreducible $C-$admissible pair for all $ k\in\Gamma_s$. \\
 The assertion 4) follows from the relation (\ref{xbar}). 
\end{proof}

\begin{cor} \label{IreCadm} The pair $(I_{re},J_{re})$ is $C-$admissible.
If $I'_{im}=\emptyset$, then $I_{re}=I$, $J_{re}=J$ 
and $\mathfrak{g}$ is finitely $\Delta^J-$graded, with grading subalgebra $\mathfrak{g}^J$.
\end{cor}
\textbf{N.B.} We have got part 2) of Theorem 2. %\ref{thmB}. 
\begin{proof} The first assertion is a consequence of Proposition \ref{ire}. By remark \ref{gap}, when 
 $I'_{im}=\emptyset$, we have $I=I_{re}$; 
 hence, by Theorem \ref{jgrad}, $\mathfrak{g}$ is finitely $\Delta^J-$graded.
\end{proof}
\begin{defn}\label{GenMaxGrad} If $I'_{im}\not=\emptyset$, then $(I,J)$ is called a  generalized $C-$admissible pair.  If  $I'_{im}=J=\emptyset$, the Kac-Moody algebra $\mathfrak{g}$ is said to be maximally   finitely $\Sigma-$graded. 
\end{defn}

\begin{cor}\label{Msym} The grading subalgebra $\mathfrak{m}$ of  $\mathfrak{g}$ is symmetrizable and the restriction to $\mathfrak{m}$ of the invariant bilinear form of $\mathfrak{g}$ is nondegenerate. 
\end{cor}
\begin{proof}
Let $(.\, , .)_a$ be the restriction to $\mathfrak{m}$ of the invariant bilinear form $(.\, , .)$ of $\mathfrak{g}$.   Recall from 
the proof of Proposition \ref{ire} that $\gamma\check{_s}=\sum_{k\in\Gamma_s}H_k$, $\forall s\in\bar I$. In particular $(\gamma\check{_s}, \gamma\check{_s})_a=\sum_{k\in\Gamma_s}(H_k,H_k)>0$. It follows that $(.\, , .)_a$ is a nondegenerate invariant bilinear form on  $\mathfrak{m}$ (see \S \ref{ibf}) and that  $\mathfrak{m}$ is symmetrizable.
\end{proof}

\begin{cor}\label{eqCartan} Let  $\mathfrak{h}^J$ be the orthogonal of  $\mathfrak{h}_J$ in  $\mathfrak{h}$.  For $k\in I_{im}'$, write $$\displaystyle\rho_a(\alpha_k)=\sum_{s\in\bar I}n_{s,k}\gamma_s.$$ For $s\in\bar I$, choose $l_s$ a representative element  of $\Gamma_s$. Then  $\mathfrak{a}/\mathfrak{c}_a$ can be viewed as the subspace of  $\mathfrak{h}^J/ \mathfrak{c}$ defined by the following relations : 
$$
\begin{array}{lll}
\langle\alpha_k,h\rangle=\langle\alpha_{l_s},h\rangle, \forall k\in\Gamma_s, \forall s\in \bar I\, & 
\\
 &  \\
 \langle\alpha_k,h\rangle=\displaystyle
\sum_{s\in\bar I}n_{s,k}\langle\alpha_{l_s},h\rangle, \forall k\in I_{im}'.&  
 \end{array}
$$
\end{cor}
\begin{proof} The subspace of  $\mathfrak{h}^J/ \mathfrak{c}$ defined by the above relations has dimension $|\bar I|$ and  contains $\mathfrak{a}/\mathfrak{c}_a$ and hence it is equal to $\mathfrak{a}/\mathfrak{c}_a$.
\end{proof}

\begin{prop}\label{aire}  
Let $(.\, , .)_a$ be the restriction to $\mathfrak{m}$ of the invariant bilinear form $(.\, , .)$ of $\mathfrak{g}$.\\
1) Let $\mathfrak{a}'=\mathfrak{a}\cap\mathfrak{m}'$  and let $\mathfrak{a}''$ be a supplementary subspace of $\mathfrak{a}'$ in $\mathfrak{a}$ which is totally isotropic relatively to $(.\, , .)_a$. Then $\mathfrak{a}''\cap\mathfrak{h}' =\{0\}$. \\
2) Let $A_{I_{re}}$ be the submatrix of $A$ indexed by $I_{re}$. Then there exists a subspace $\mathfrak{h}_{I_{re}}$ of $\mathfrak{h}$ containing $\mathfrak{a}$ such that $(\mathfrak{h}_{I_{re}}$, $\Pi_{I_{re}},\Pi\check{_{I_{re}}})$ is a realization of $A_{I_{re}}$.
 In particular, the Kac-Moody subalgebra  $\mathfrak{g}(I_{re})$ associated to this realization (in \ref{realj}) contains the grading subalgebra $\mathfrak{m}$. \\
 3) The Kac-Moody algebra $\mathfrak{g}(I_{re})$ is finitely $\Delta({I_{re}})^{J_{re}}-$graded and its grading subalgebra is the subalgebra $\mathfrak g({I_{re}})^{J_{re}}$ associated to the $C-$admissible pair $(I_{re},J_{re})$ as in Proposition \ref{gJ}.
 \\
 4) The Kac-Moody algebra $\mathfrak g({I_{re}})^{J_{re}}$ contains $\mathfrak m$.
\end{prop} 

\begin{proof} $\,$\\
1) Recall that the center $\mathfrak{c}_{a}$ of $\mathfrak{m}$ is contained in the center $\mathfrak{c}$ of $\mathfrak{g}$.  Since
$\mathfrak{h}'=\mathfrak{c}^{\perp}$ and   
$\mathfrak{c}_{a}$ is in duality with $\mathfrak{a}''$ relatively to  $(.\, , .)_a$, we deduce that $\mathfrak{a}''\cap\mathfrak{h}' =\{0\}$.\\
2) From the proofs of \ref{Msym} and \ref{ire} we get $\gamma_s^\vee=\sum_{k\in\Gamma_s}\,H_k\in\sum_{k\in\Gamma_s}\,\mathfrak h_{I_k}=\mathfrak h'_{I_{re}}$. So
 $\mathfrak{c}_{a}\subset\mathfrak{a}'\subset\mathfrak{h'}_{I_{re}}\subset\mathfrak{h}'$. It follows that  $(\mathfrak{h}_{I_{re}}'+\mathfrak{h}^{I_{re}})$ is contained in  $\mathfrak{c}_{a}^{\perp}$ the orthogonal subspace of $\mathfrak{c}_{a}$ in $\mathfrak{h}$. Since $\mathfrak{a}''\cap\mathfrak{c}_{a}^{\perp}=\{0\}$, one can choose a supplementary subspace $\mathfrak{h}_{I_{re}}''$ of $(\mathfrak{h}_{I_{re}}'+\mathfrak{h}^{I_{re}})$ containing $\mathfrak{a}''$. Let $\mathfrak{h}_{I_{re}}=\mathfrak{h}_{I_{re}}'\oplus\mathfrak{h}_{I_{re}}''$, then, by Proposition \ref{realj},  $(\mathfrak{h}_{I_{re}}$, $\Pi_{I_{re}},\Pi\check{_{I_{re}}})$ is a realization of $A_{I_{re}}$.\\
3) As in Corollary \ref{IreCadm}, assertion 3) is a simple consequence of Theorem \ref{jgrad}.\\
4) The algebra $\mathfrak a$ is in $\mathfrak h_{I_{re}}\cap\Pi_J^\perp=(\mathfrak h_{I_{re}})^{J_{re}}$.
By the proof of Proposition \ref{ire}, for $s\in\overline I$, $\overline X_s$ and $\overline Y_s$ are linear combinations of the elements in $\{E_k,F_k\mid k\in\Gamma_s\}\subset \mathfrak g({I_{re}})^{J_{re}}$.
 Hence $\mathfrak g({I_{re}})^{J_{re}}$ contains all generators of $\mathfrak m$.
\end{proof}

\begin{lem}\label{sg} Let $\mathfrak{l}$ be a Kac-Moody subalgebra of $\mathfrak{g}$ containing $\mathfrak{m}$. Then $\mathfrak{l}$ is finitely $\Sigma-$graded. In particular,  the Kac-Moody subalgebra  $\mathfrak{g}(I_{re})$ or $\mathfrak g({I_{re}})^{J_{re}}$  is finitely $\Sigma-$graded.
\end{lem} 
\textbf{N.B.} Proposition \ref{aire} and Lemma \ref{sg} finish the proof  of Theorem 2. %\ref{thmB}. 
\begin{proof} Recall that the Cartan subalgebra $\mathfrak{a}$ of $\mathfrak{m}$ is ad$_{\mathfrak{g}}-$diagonalizable. Since $\mathfrak{l}$ is ad$(\mathfrak{a})-$invariant, one has  $\mathfrak{l}=\displaystyle\sum_{{\gamma}\in\Sigma \cup\{0\}}V_{{\gamma}}\cap\mathfrak{l}$. By assumption  $\{0\}\not=\mathfrak{m}_{\gamma}\subset V_{{\gamma}}\cap\mathfrak{l}$ for all $\gamma\in\Sigma$;   hence, we deduce that $\mathfrak{l}$ is finitely $\Sigma-$graded.
\end{proof}

\begin{prop}\label{imempty} If $\mathfrak{g}$ is of finite, affine or hyperbolic type, then $I'_{im}=\emptyset$ and $(I,J)$ is a   $C-$admissible pair.
\end{prop}
\begin{proof} The result is trivial if  $\mathfrak{g}$ is of finite type.  Suppose $I'_{im}\not=\emptyset$ for one of  the other cases. If  $\mathfrak{g}$ is  affine, then $I_{re}$ is of finite type and by Lemma \ref{aire}  , $\mathfrak{m}$ is contained in the finite dimensional semi-simple Lie algebra $\mathfrak{g}(I_{re})$. This contradicts the fact that $\mathfrak{m}$ is, as $\mathfrak{g}$, of affine type (see Proposition \ref{st}). If  $\mathfrak{g}$ is  hyperbolic, then it is Lorentzian and perfect (cf. section \ref{1.1}), 
and by Lemma \ref{sg}, $\mathfrak{g}(I_{re})$ is a finitely $\Sigma-$graded subalgebra of $\mathfrak{g}$. As $I_{re}$ is assumed to be a proper subset of $I$,  $\mathfrak{g}(I_{re})$ is  of finite or affine type. This contradicts Proposition \ref{st},  since $\mathfrak{m}$ should be Lorentzian (cf. \ref{st}).  %\marginpar{modif}
Hence, $I'_{im}=\emptyset$ in the two last cases. 
\end{proof}

\begin{prop}\label{sth} If $\mathfrak{g}$ is of  hyperbolic type, then the grading subalgebra $\mathfrak{m}$ is also of  hyperbolic type.
\end{prop}
\begin{proof} Recall that in this case, $I_{re}=I$ (see Proposition \ref{imempty} and Corollary \ref{IreCadm}).  Let ${\bar I}^1$ be a proper subset of $\bar I$ and suppose that ${\bar I}^1$ is connected. Let  $I^1=\displaystyle\mathop{\cup}_{s\in{\bar I}^1} (\mathop{\cup}_{k\in{\Gamma}_s}I_k)$. Then, $I^1$ is a proper subset of  $I$. We may assume that the subalgebra $\mathfrak{m}({\bar I}^1)$ of $\mathfrak{m}$ is contained in $\mathfrak{g}(I^1)$. Let $\Sigma^1:=\Sigma({\bar I}^1)$ be the root system of $\mathfrak{m}({\bar I}^1)$. Then, it is not difficult to check that $\mathfrak{g}(I^1)$ is $\Sigma^1-$graded. Since  $\mathfrak{g}(I^1)$ is of finite or affine type, we deduce, by Proposition \ref{st}, that $\mathfrak{m}({\bar I}^1)$ is of finite or affine type. Hence,  $\mathfrak{m}$ is hyperbolic.
\end{proof}

\begin{prop}\label{imempty2} If   $I'_{im}=\emptyset$, then $\mathfrak{g}(I_{re})=\mathfrak{g}$ and the  $C-$admissible subalgebra $\mathfrak{g}^{J}$ is  maximally finitely $\Sigma-$graded, with grading subalgebra $\mathfrak m$.
\end{prop}
\begin{proof} This result is  due to J. Nervi (\cite{nerv2}; Thm 2.5.10) for the affine case; it follows from the facts that $V_0\cap \mathfrak{g}^{J}=\mathfrak{h}^{J}$ and $\mathfrak m\subset\mathfrak g^J$ (see Prop. \ref{aire}).
\end{proof}
\begin{cor}\label{classif} If $\mathfrak{g}$ is of finite, affine or hyperbolic type, the problem of classification of finite gradations of $\mathfrak{g}$  comes down first to classify the $C-$admissible pairs $(I,J)$ of  $\mathfrak{g}$ and then the maximal finite gradations of the corresponding admissible algebra $\mathfrak{g}^J$. 
\end{cor}
\begin{proof}
This follows from Proposition \ref{imempty}, Proposition \ref{imempty2} and Lemma \ref{trans}.
\end{proof}

\section{Maximal gradations}\label{s4} 

We assume now that $\mathfrak g$ is maximally finitely $\Sigma-$graded. We keep the notations in section \ref{GenGrad} but we have $J=I'_{im}=\emptyset$. So $\overline I$ is a quotient of $I$, with quotient map $\rho$ defined by $\rho_a(\alpha_k)=\gamma_{\rho(k)}$. For $s\in\overline I$, $\Gamma_s=\rho^{-1}(\{s\})$.

\begin{prop}\label{4.1} $\,$\\
1) If $k\neq l\in I$ and $\rho(k)=\rho(l)$, then there is no link between $k$ and $l$ in the Dynkin diagram of $A$: $\alpha_k(\alpha_l^\vee)=\alpha_l(\alpha_k^\vee)=0$ and $(\alpha_k,\alpha_l)=0$.
\\
2) $\mathfrak a\subset \{h\in\mathfrak h\mid\alpha_k(h)=\alpha_l(h)\text{ whenever } \rho(k)=\rho(l)\}$.
\\
3) For good choices of the simple coroots and Chevalley generators $(\alpha_k^\vee,e_k,f_k)_{k\in I}$ in $\mathfrak g$ and $(\gamma_s^\vee,\overline X_s,\overline Y_s)_{s\in\overline I}$ in $\mathfrak m$, we have $\gamma_s^\vee=\sum_{k\in\Gamma_s}\,\alpha_k^\vee$, $\overline X_s=\sum_{k\in\Gamma_s}\,e_k$ and $\overline Y_s=\sum_{k\in\Gamma_s}\,f_k$.
\\
4) In particular, for $s,t\in\overline I$, we have $\gamma_s(\gamma_t^\vee)=\sum_{k\in\Gamma_t}\,\alpha_i(\alpha_k^\vee)$ for any $i\in\Gamma_s$.
\end{prop}

\begin{proof} Assertions 1) and 2) are proved in \ref{ire} and \ref{eqCartan}.
 For $i\in\Gamma_s$, $\gamma_s=\rho_a(\alpha_i)$ is the restriction of $\alpha_i$ to $\mathfrak a$; so 4) is a consequence of 3).
 
 For 3) recall the proof of Proposition \ref{ire}. The $\mathfrak{sl}_2-$triple $(\overline X_s,\gamma_s^\vee,\overline Y_s)$ may be written $\gamma_s^\vee=\sum_{k\in\Gamma_s}\,H_k$, $\overline X_s=\sum_{k\in\Gamma_s}\,E_k$ and $\overline Y_s=\sum_{k\in\Gamma_s}\,F_k$ where $(E_k,H_k,F_k)$ is an $\mathfrak{sl}_2-$triple in $\mathfrak g(I_k)$, with $\alpha_k(H_k)=2$.
 But now $J=I'_{im}=\emptyset$, so $I_k=\{k\}$ and $\mathfrak g(I_k)=\Complex e_k\oplus\Complex\alpha_k^\vee\oplus\Complex f_k$, hence the result.
 \end{proof}
 
 \bigskip So the grading subalgebra $\mathfrak m$ may be entirely described by the quotient map $\rho$.
 \par We look now to the reciprocal construction.  
 \bigskip
\par So $\mathfrak g$ is an indecomposable and symmetrizable Kac-Moody algebra associated to a generalized Cartan matrix $A=(a_{i,j})_{i,j\in I}$. We consider a quotient $\overline I$ of $I$ with quotient map $\rho:I\to\overline I$ and fibers $\Gamma_s=\rho^{-1}(\{s\})$ for $s\in\overline I$. We suppose that $\rho$ is an admissible quotient i.e. that it satisfies the following two conditions: %\marginpar{r-modif}
\medskip
\par
(MG1)  If $k\neq l\in I$ and $\rho(k)=\rho(l)$, then $a_{k,l}=\alpha_l(\alpha_k^\vee)=0$.

(MG2) If $s\neq t\in\overline I$, then $\overline a_{s,t}:=\displaystyle\sum_{i\in\Gamma_s}\,a_{i,j}=\sum_{i\in\Gamma_s}\,\alpha_j(\alpha_i^\vee)$ is independent of the choice of $j\in\Gamma_t$.
\medskip
\begin{prop}\label{4.2} The matrix $\overline A=(\overline a_{s,t})_{s, t\in\overline I}$ is an indecomposable generalized Cartan matrix.
\end{prop}
\begin{proof} Let $s\neq t\in\bar I$ and let  $j\in\Gamma_t$. By (MG1)  one has $\bar a_{t,t}=\sum_{i\in\Gamma_t}a_{i,j}= a_{j,j}=2$, and by (MG2) $\overline a_{s,t}:=\sum_{i\in\Gamma_s}\, a_{i,j}\in \Zmius$ ($\forall j\in\Gamma_t$). Moreover, $\overline a_{s,t}=0$  if and only if $a_{i,j}= 0\, (=a_{j,i})$, $\forall (i,j)\in\Gamma_s\times\Gamma_t$. It follows that $\overline a_{s,t}=0$  if and only if $\overline a_{t,s}=0$, and $\bar A$ is a generalized Cartan matrix. Since $A$ is indecomposable, $\bar A$ is also indecomposable.  
\end{proof}
\bigskip
\par Let $\mathfrak h^\Gamma=\{h\in\mathfrak h\mid\alpha_k(h)=\alpha_l(h)\text{ whenever } \rho(k)=\rho(l)\}$, $\gamma_s^\vee=\sum_{k\in\Gamma_s}\,\alpha_k^\vee$ and $\mathfrak a'=\oplus_{s\in\overline I}\,\Complex\gamma_s^\vee\subset\mathfrak h^\Gamma$. 
We may choose a subspace $\mathfrak a''$ in $\mathfrak h^\Gamma$ such that $\mathfrak a''\cap\mathfrak a'=\{0\}$, the restrictions $\overline\alpha_i=:\gamma_{\rho(i)}$ to $\mathfrak a=\mathfrak a'\oplus\mathfrak a''$ of the simple roots $\alpha_i$ (corresponding to different $\rho(i)\in \overline I$) are linearly independent and $\mathfrak a''$ is minimal for these two properties.

\begin{prop}\label{4.3} $(\mathfrak a,\{\gamma_s\mid{s\in\overline I}\},\{\gamma_s^\vee\mid{s\in\overline I}\})$ is a realization of $\overline A$.
\end{prop}
\begin{proof} Let $\ell$ be the rank of $\overline A$. Note that $\mathfrak a$ contains $\mathfrak a'=\oplus_{s\in\overline I}\,\Complex\gamma_s^\vee$;  the family  $(\gamma_s)_{s\in\bar I}$ is  free  in the dual space $\mathfrak a^*$ of $\mathfrak a$ and  satisfies $\langle\gamma_t,\gamma_s^{\vee}\rangle=\bar a_{s,t}$, $\forall s,t\in\bar I$. It follows that $\dim(\mathfrak a)\geq 2|\bar I|-\ell$ (see [\cite{carter}; Prop. 14.1] or [\cite{vk}; Exer. 1.3]). As $\mathfrak a$ is minimal, we have $\dim(\mathfrak a)= 2|\bar I|-\ell$ (see  [\cite{carter};  Prop. 14.2] for minimal realization). Hence $(\mathfrak a,\{\gamma_s\mid{s\in\overline I}\},\{\gamma_s^\vee\mid{s\in\overline I}\})$ is a (minimal) realization of $\overline A$.  
\end{proof}
\par We note $\Delta^\rho=\Sigma\subset\oplus_{s\in\overline I}\,\Z\gamma_s$ the root system associated to this realization.
\par We define now $\overline X_s=\sum_{k\in\Gamma_s}\,e_k$ and $\overline Y_s=\sum_{k\in\Gamma_s}\,f_k$. Let $\mathfrak m=\mathfrak g^\rho$  be the Lie subalgebra of $\mathfrak g$  generated by $\mathfrak a$ and the elements $\overline X_s$, $\overline Y_s$ for $s\in\overline I$.
\begin{prop}\label{4.4} The Lie  subalgebra $\mathfrak m=\mathfrak g^\rho$   is the Kac-Moody algebra associated to the realization 
$(\mathfrak a,\{\gamma_s\mid{s\in\overline I}\},\{\gamma_s^\vee\mid{s\in\overline I}\})$ of $\overline A$.  Moreover, $\mathfrak g$ is an integrable $\mathfrak g^\rho-$module with finite multiplicities.
\end{prop}

\begin{proof}
Clearly, the following relations hold  in the Lie subalgebra $\mathfrak g^\rho$ : 
 
$$
\begin{array}{lll}
[\mathfrak{a}, \mathfrak{a}]=0, & [\overline X_s,\overline Y_t]=\delta_{s,t}\gamma_s^\vee & (s,t \in \bar I);\\
% \, & \,  & \,  \\
 
[a, \overline X_s] = \langle\gamma_s, a\rangle \overline X_s, & [a,\overline Y_s]= -\langle\gamma_s, a\rangle \overline Y_s & (a \in \mathfrak{a},  s\in \bar I).\\
 &  &
 \end{array}
$$
For the Serre's relations, one has : $$1-\overline a_{s,t}\geq 1-a_{i,j}, \; \forall (i,j)\in\Gamma_s\times\Gamma_t.$$ 
In particular, one can see, by induction on $|\Gamma_s|$, that :$$(\text{ad}\overline X_s)^{1-\overline a_{s,t}}(e_j) =(\sum_{i\in\Gamma_s}\text{ad}e_i)^{1-\overline a_{s,t}}(e_j)=0, \; 
\forall j\in \Gamma_t.$$ Hence $$(\text{ad}\overline X_s)^{1-\overline a_{s,t}}(\overline X_t)=0, \; \forall s, t \in\overline I,$$ and in the same way we obtain that :  
$$(\text{ad}\overline Y_s)^{1-\overline a_{s,t}}(\overline Y_t)=0, \; \forall s, t \in\overline I.$$
 It follows 
that  $\mathfrak g^\rho$ is a quotient of the Kac-Moody algebra  $\mathfrak{g}(\overline A)$ associated to $\overline A$ and $(\mathfrak a,\{\gamma_s\mid{s\in\overline I}\},\{\gamma_s^\vee\mid{s\in\overline I}\})$ in which the Cartan subalgebra $\mathfrak a$ of  $\mathfrak{g}(\overline A)$ is embedded. By \cite[1.7]{vk}  $\mathfrak g^\rho$ is equal to  $\mathfrak{g}(\overline A)$.\\ 
It's clear that $\mathfrak g$ is an integrable $\mathfrak g^\rho-$module with finite dimensional weight spaces relative to the adjoint action of $\mathfrak a$, since for $\alpha=\sum_{i\in I}n_i\alpha_i\in \Delta^+$, its restriction to $\mathfrak a$, is given by  \begin{equation}\label{rhoa}\rho_a(\alpha)=\sum_{s\in\bar I}\big(\sum_{i\in\Gamma_s}n_i\big)\,\gamma_s
\end{equation} %\marginpar{hbm-modif}
\end{proof}

\begin{prop}\label{4.5} The Kac-Moody algebra $\mathfrak g$ is maximally finitely $\Delta^\rho-$graded with grading subalgebra $\mathfrak g^\rho$.
\end{prop}
\begin{proof}  As in Theorem \ref{jgrad}, we will see that $\rho_a(\Delta^+)\subset Q_+^\Gamma:=\displaystyle\mathop{\oplus}_{s\in\bar I}\Z^+\gamma_s$ satisfies, as $\Sigma^+=\Delta^\rho_+$,  the following conditions : 
\\
(i) $\gamma_s\in\rho_a(\Delta^+)\subset Q_+^\Gamma$, $2\gamma_s\notin\rho_a(\Delta^+)$, $\forall s\in \bar I$.
\\
(ii) if $\gamma\in\rho_a(\Delta^+)$, $\gamma\not=\gamma_s$, then the set $\{\gamma+k\gamma_s; k\in\Z\}\cap \rho_a(\Delta^+)$ is a string $\{\gamma-p\gamma_s, ...., \gamma+q\gamma_s\}$, where $p,q\in\Zplus$ and $p-q=\langle \gamma,\gamma_s^\vee\rangle$;
\\
(iii) if $\gamma\in\rho_a(\Delta^+)$, then supp$(\gamma)$ is connected.\\
 Clearly $\{\gamma_s\mid{s\in\overline I}\}\subset\rho_a(\Delta_+)\subset Q^\Gamma_+$. 
For $\alpha\in \Delta$ and $s\in \bar I$, the condition $\rho_a(\alpha)\in\Nat\gamma_s$ implies $\alpha\in \Delta(\Gamma_s)^+=\{\alpha_i; \; i\in\Gamma_s\}$ [see (\ref{rhoa})]. 
 It follows that $2\gamma_s\notin  \rho_a(\Delta_+)$ and (i) is satisfied.  By Proposition \ref{4.4}, $\mathfrak{g}$ is an integrable $\mathfrak g^\rho-$module with finite multiplicities. Hence,  the propriety (ii) follows from [\cite{vk}; prop.3.6]. 
Let  $\alpha\in\Delta_+$ and let   $s,t\in$ supp$(\rho_a(\alpha))$. By (\ref{rhoa}) there exists  $(k,l)\in\Gamma_s\times\Gamma_t$ such that $k,l \in$ supp$(\alpha)$, which is connected. Hence there exist $i_0=k, i_1 ,  ...., i_{n+1}= l$ such that $\alpha_{i_j}\in$ supp$(\alpha)$, $j=0, 1, ...., n+1$, and for  $j=0,1, ..., n$,  $i_j$ and $i_{j+1}$ are linked  relative to the generalized Cartan matrix $A$. In particular, $\rho(i_j)\neq \rho(i_{j+1})\in$ supp$(\rho_a(\alpha))$ and they  are linked relative to the generalized Cartan matrix $\overline A$, $j=0,1, ..., n$, with  $\rho(i_0)=s$ and  $\rho(i_{n+1})=t$.  Hence the connectedness of   supp$(\rho_a(\alpha))$ relative to $\overline A$. 
It follows that $\rho_a(\Delta^+)=\Delta^\rho_+$ and hence  $\rho_a(\Delta)=\Delta^\rho$ (see \cite{vk}, Ex. 5.4). In particular, $\mathfrak g$ is  finitely $\Delta^\rho-$graded with $J=\emptyset=I'_{im}$. \\ 
\end{proof}
\begin{cor}\label{4.6} The restriction to $\mathfrak m=\mathfrak g^\rho$ of the invariant bilinear form $(.\, , \, .)$ of  $\mathfrak g$ is  nondegenerate. In particular, the generalized Cartan matrix $\overline A$ is symmetrizable of the same type as $A$.
\end{cor}
 
\begin{proof} The first part of the corollary  follows form Proposition \ref{4.5} and Corollary \ref{Msym}. 
The second part follows form Proposition \ref{st}.
\end{proof}

\begin{rem}\label{4.7}
 The map $\rho$ coincides with the map (also denoted $\rho$) defined at the beginning of this section using the maximal gradation of Proposition \ref{4.5}.
 Conversely Proposition \ref{4.1} tells that, for a general maximal gradation, $\rho$ is admissible and $\mathfrak m=\mathfrak g^\rho$ for good choices of the Chevalley generators.
 So we get a good correspondence between maximal gradations and admissible quotient maps.
 \end{rem}

\section{An example}\label{s5}

The following example shows that generalized  $C-$admissible pairs do %always
 exist. It shows in particular  that, for a generalized  $C-$admissible pair $(I,J)$,  $J^{\circ}$ may be nonempty and $I_{re}$ may be non connected. Moreover, the Kac-Moody algebra may  be not graded by the root system of $\mathfrak{g}(I_{re})$.\\
Gradations revealing generalized $C-$admissible pairs will be studied in a forthcoming paper.  

\begin{exa}
Consider the Kac Moody algebra  $\mathfrak{g}$ corresponding to the indecomposable and symmetric generalized Cartan matrix $A$  : 
$$A=\left(\begin{array}{cccccc} 
2 & -3 & -1& 0 & 0 & 0 \\ 
-3 & 2 & -1 & 0 & 0 & 0 \\ 
-1 & -1 & 2 & -1 & -1 & -1 \\ 
0 & 0 & -1 & 2 & 0 & 0 \\ 
0 & 0 & -1 & 0 & 2& -3\\ 
0 & 0 & -1 & 0 & -3& 2 
\end{array}      \right)$$
with the corresponding Dynkin diagram :
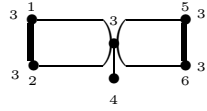
\begin{center} $\,$
\hbox to 4,2cm{\lower -9pt
\vtop {\offinterlineskip \hbox to 1.2cm{\lower 2pt\hbox{\hglue -1pt${}^{_3}\displaystyle
\mathop{\bullet}^{_1}$}
\hglue -4pt\hrulefill  }
\vskip -1pt
\hbox to 0.5cm{\hfill \vrule height 14pt  width 2pt\hfill}
\hbox to 1.2cm{ 
\lower 2pt
\hbox{\hglue -4pt${}_{_3}\displaystyle\mathop{\bullet}_{^2}$} 
\hglue -4pt\hrulefill } }\lower 2pt\hbox{\hglue 2pt
\vtop{ \offinterlineskip \hbox{$\displaystyle\mathop{\bullet}^{^3}$}
\vskip -1pt
\hbox to 5pt{\hfill\vrule height 10pt width 0,3pt\hfill}
\hbox{$\displaystyle\mathop{\bullet} _{_4}$\hfill}
} }
\hglue -44pt \lower 2pt\hbox{$\Big)$\hglue-2pt$\Big($}\lower -9pt
\vtop {\offinterlineskip \hbox to 1.2cm{
\hglue -4pt\hrulefill\lower 2pt\hbox{\hglue -1pt$\displaystyle
\mathop{\bullet}^{_5}{}^{_3}$}  }
\vskip -1pt
\hbox to 1.5cm{\hfill \vrule height 14pt  width 2pt\hfill}
\hbox to 1.2cm{
\hglue -4pt\hrulefill
\lower 2pt
\hbox{\hglue -1pt$\displaystyle\mathop{\bullet}_{^6}{}_{^3}$}  } } }
\end{center}
Note that $\det(A)=275$ and the the symmetric submatrix of $A$ indexed by $\{1,2, 4,5,6\}$ has signature $(+++,- -)$. Since $\det(A)>0$, the matrix $A$ should have signature $(++++,- -)$. Let $\Sigma$ be the root system associated to the strictly hyperbolic generalized Cartan matrix $ \left(\begin{array}{cc} 
2 & -3   \\ 
-3 & 2 
\end{array}      \right)$, the corresponding Dynkin diagram is the following : 
\begin{center} 
 $H_{3,3}\qquad$   \hbox{\offinterlineskip\lower 2pt
\hbox{\hglue -2pt \hbox{$\displaystyle\mathop{\bullet}_{^3}^1$}\hglue -3pt \vbox{  {
\hrule height 2pt width 1.5cm}\vskip 1.4pt}
\hglue -5pt \hbox{$\displaystyle\mathop{\bullet}_{^3}^2$} } }  
\end{center}  We will see that $\mathfrak{g}$ is finitely $\Sigma-$graded and describe the corresponding generalized 
$C-$admissible pair.
\end{exa}
\textbf{1)} Let $\tau $ be the involutive diagram automorphism of $\mathfrak{g}$ such that $\tau(1)=5$, $\tau(2)=6$ and $\tau$ fixes the other vertices. Let $\sigma_n'$ be the normal semi-involution of $\mathfrak{g}$ corresponding to the split real form of $\mathfrak{g}$. Consider the quasi-split real form  $\mathfrak{g}_{\Real}^{1}$ associated to the semi-involution $\tau\sigma_n'$ (see \cite{b3r} or \cite{bm}). Then $\mathfrak{t}_{\Real}:=\mathfrak{h}_{\Real}^{\tau}$ is a maximal split toral subalgebra of $\mathfrak{g}_{\Real}^{1}$. 
The corresponding restricted root system $\Delta':=\Delta(\mathfrak{g}_{\Real}, \mathfrak{t}_{\Real})$ is reduced and the corresponding generalized Cartan matrix $A'$ is given by : 
$$A'=\left(\begin{array}{cccc} 
2 & -3 & -2& 0  \\ 
-3 & 2 & -2 & 0  \\ 
-1 & -1 & 2 & -1 \\ 
0  & 0  & -1 & 2  
\end{array}      \right)$$
with the corresponding Dynkin diagram  :  
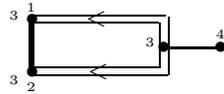
\begin{center} $\,$
\hbox{\lower -16pt
\hbox{\lower 16.5pt\hbox to 2pt{\hfill\vrule height 17.5pt width 2pt}
\hglue -12.7pt
\vtop{\offinterlineskip
\hbox to 2cm{\lower 0.5pt \hbox{${}^{_3}\displaystyle\mathop{\bullet}^{_1}$}
\hglue -5.5pt
\vbox{  {\hrule height 0,3pt width 1,8cm}
\vskip 2.5pt {\hrule height 0,3pt width 1,7cm}\vskip 0.5pt}\hglue -20pt \lower 0.5pt
\hbox{\hglue -10pt $<$}
\hglue -2.5pt\lower 0.5pt
\hbox{$\qquad$}
}
\hbox to 0pt{\hfill \vrule height 14pt width 0pt\hglue 0pt}
%\hbox to 0pt{\hfill \vrule height 12pt width 0pt\hfill}
\hbox to 2cm{\lower 0.5pt
\hbox{${}_{_3}\displaystyle\mathop{\bullet}_{^2}$}\hglue -2.5pt 
\vbox{  {\hrule height 0,3pt width 1,7cm}
\vskip 2.5pt {\hrule height 0,3pt width 1,8cm}\vskip 0.5pt}\hglue -20pt \lower 0.5pt
\hbox{\hglue -10pt $<$}
\hglue -2.5pt\lower 0.5pt 
\hbox{}
}  }
\hglue -3.5pt\lower 16pt\hbox to 1.5pt{\vrule height 16.5pt width 0,3pt}\hglue 2pt\lower 19pt\hbox to 1.5pt{\vrule height 22pt width 0,3pt\hfill}  }
 \hglue -17pt\lower -7pt\hbox to 1.2cm{\lower 2pt\hbox{${{}^{_3}\bullet}$}
\hglue -4pt\hrulefill\lower 2pt\hbox{\hglue -1pt$\displaystyle
\mathop{\bullet}^{_4}$}  }
}
\end{center}
Following N. Bardy [\cite{ba}, \S 9], there exists a split real Kac-Moody subalgebra  $\mathfrak{m}_{\Real}^1$ of $\mathfrak{g}_{\Real}^1$ containing $\mathfrak{t}_{\Real}$ such that $\Delta'= \Delta(\mathfrak{m}_{\Real}^1,\mathfrak{t}_{\Real})$. It follows that  $\mathfrak{g}$ is  finitely $\Delta'- $graded.

\medskip
\textbf{2)} Let $\mathfrak{m}^1:=\mathfrak{m}_{\Real}^1\otimes\Complex$ and $\mathfrak{t}:=\mathfrak{t}_{\Real}\otimes\Complex$. Denote by  $\alpha'_i:=\alpha_i/\mathfrak{t}$, $i=1,2,3,4$. Put $\alpha'\check{_1}=\alpha\check{_1}+\alpha\check{_5}$, $\alpha'\check{_2}=\alpha\check{_2}+\alpha\check{_6}$,  $\alpha'\check{_3}=\alpha\check{_3}$ and $\alpha'\check{_4}=\alpha\check{_4}$.  Let $I^1:=\{1,2,3,4\}$, then $(\mathfrak{t},\Pi'=\{\alpha'_i, \; i\in I^1\}, \Pi'^\vee=\{\alpha'\check{_i}, \; i\in I^1\})$ is a realization of $A'$ which is  symmetrizable and Lorentzian.

Let $\mathfrak{m}$ be the Kac-Moody  subalgebra of $\mathfrak{m}^1$ corresponding to the  submatrix $\bar A$ of $A'$ indexed by $\{1,2\}$. Thus 
$\bar A=\left(\begin{array}{cc} 
2 & -3  \\ 
-3 & 2  
\end{array}      \right)$ is strictly hyperbolic.  Let $\mathfrak{a}:= \Complex\alpha'\check{_1} \oplus \Complex\alpha'\check{_2}$ be  the standard Cartan subalgebra of $\mathfrak{m}$ and let $\Sigma=\Delta(\mathfrak{m},\mathfrak{a})$. For $\alpha'\in \mathfrak{t}^*$, denote by $\rho_1(\alpha')$ the restriction of $\alpha'$ to
 $\mathfrak{a}$. Put $\gamma_s =\rho_1(\alpha'_s)$,   $\gamma\check{_s}=\alpha'\check{_s}$, $s=1,2$. Then $\Pi_a=\{\gamma_1, \gamma_2\}$  is the standard root basis 
 of $\Sigma$. One can see easily  that $\rho_1(\alpha'_4)=0$ and $\rho_1(\alpha'_3)= 2(\gamma_1+\gamma_2)$ is a strictly positive imaginary root of $\Sigma$.\\
 Now  we will show that $\mathfrak{m}^1$ is finitely $\Sigma-$graded.\\
Let $(.\, , .)_1$ be the normalized invariant bilinear form on $\mathfrak{m}^1$ such that  short real roots have length 1 and long real roots have square length 2. 
Then there exists a positive rational $q$ such that the restriction of  $(.\, , .)_1$ to $\mathfrak{t}$ has the matrix $B_1$ in the basis $\Pi'\check{_{\,}}$, where :
$$B_1= q\left(\begin{array}{cccc} 
2 & -3 & -1& 0  \\ 
-3 & 2 & -1 & 0  \\ 
-1 & -1 & 1 & -1/2 \\ 
0  & 0  & -1/2 & 1  
\end{array}      \right)$$
By duality, the restriction of  $(.\, , .)_1$ to $\mathfrak{t}$ induces a nondegenerate symmetric bilinear form on   $\mathfrak{t}^*$  (see \cite{vk}; \S 2.1) such that its matrix $B'_1$ in the basis $\Pi'$, is the following :
$$B'_1=q^{-1}\left(\begin{array}{cccc} 
2 & -3 & -2& 0  \\ 
-3 & 2 & -2 & 0  \\ 
-2 & -2 & 4 & -2 \\ 
0  & 0  & -2 & 4  
\end{array}      \right)$$
Hence,  $q $   equals $2$.
 \\
Note that for $\alpha'=\sum_{i=1}^4 n_i\alpha'_i\in\Delta'^+$, we have that 
\begin{equation}\label{lalph1} 
(\alpha',\alpha')_1= n_1^2+n_2^2+2n_3^2+2n_4^2 -3n_1n_2 -2n_1n_3 -2n_2n_3 -2n_3n_4.
\end{equation}
We will show that $\rho_1(\Delta'^+)=\Sigma^+\cup\{0\}$.  Note that $\Sigma$ can be identified with  $\Delta'\cap( \Z\alpha'_1+ \Z\alpha'_2)$; hence  $\rho_1$ is injective on $\Sigma$ and $\Sigma^+\subset\rho_1(\Delta'^+)$. \\
Let $(.\, , .)_a$ be the normalized invariant bilinear form on $\mathfrak{m}$ such that  all real roots have length 2. Then the restriction of  $(.\, , .)_a$ to $\mathfrak{a}$ has the matrix $B_a$ in the basis $\Pi\check{_a}=\{\gamma\check{_1}, \gamma\check{_2}\}$, where :
$$B_a= \left(\begin{array}{cc} 
2 & -3   \\ 
-3 & 2 
\end{array}      \right)$$
Since $\bar A$ is symmetric, the nondegenerate symmetric bilinear form, on   $\mathfrak{a}^*$, induced by the restriction of  $(.\, , .)_a$ to $\mathfrak{a}$, has the same matrix $B_a$ in the basis $\Pi_a$. In particular, we have that :
$$(\rho_1(\alpha'),\rho_1(\alpha'))_a= 2[(n_1+2n_3)^2+(n_2+2n_3)^2-3(n_1+2n_3)(n_2+2n_3)], $$
since $\rho_1(\alpha')=(n_1+2n_3)\gamma_1+(n_1+2n_3)\gamma_2$.\\
Using (\ref{lalph1}), it  is not difficult to check that 
\begin{equation}\label{r1alph} 
(\rho_1(\alpha'),\rho_1(\alpha'))_a= 
2[(\alpha',\alpha')_1 -(n_3-n_4)^2 -5n_3^2-n_4^2]
\end{equation}
Suppose $n_3=0$, then, since supp$(\alpha')$ is connected, we have that  $\alpha'=n_1\alpha'_1+n_2\alpha'_2$ or  $\alpha'=\alpha'_4$. Hence $\rho_1(\alpha') = n_1\gamma_1+n_2\gamma_2\in\Sigma$ or $\rho_1(\alpha')=0$.\\
Suppose $n_3\not=0$, then, since $(\alpha',\alpha')_1\leq 2$, one can see, using  (\ref{r1alph}),  that  $$(\rho_1(\alpha'),\rho_1(\alpha'))_a <0.$$ As $\Sigma$ is hyperbolic and $\rho_1(\alpha')\in \Nat\gamma_1+\Nat\gamma_2$, we deduce that $\rho_1(\alpha')$ is a positive imaginary root of $\Sigma$ (see \cite{vk}; Prop. 5.10). It follows that $\rho_1(\Delta'^+)=\Sigma^+\cup\{0\}$. \\
 To see that   
$\mathfrak{m}^1$ is finitely $\Sigma-$graded, it suffices to prove that, for $\gamma=m_1\gamma_1+m_2\gamma_2\in\Sigma^+\cup\{0\}$, the set $\{\alpha'\in\Delta'^+, \rho_1(\alpha')=\gamma\}$ is finite. 
Note that if $\alpha'=\sum_{i=1}^4 n_i\alpha'_i\in\Delta'^+$ satisfying $\rho_1(\alpha')=\gamma$, then $n_i+2n_3=m_i$, $i=1, 2$. In particular, there are only finitely many  possibilities for $n_i$, $i=1, 2, 3$. The same argument as the one %\marginpar{modif} 
used in the proof of Proposition \ref{finmult} shows   also that there are only finitely many possibilities for $n_4$.

\medskip
\textbf{3)} Recall that $\mathfrak{m}\subset\mathfrak{m}^1\subset\mathfrak{g}$. The fact that $\mathfrak{g}$ is finitely $\Delta'-$graded with grading subalgebra $\mathfrak{m}^1$ 
and $\mathfrak{m}^1$ is finitely $\Sigma-$graded  implies that $\mathfrak{g}$ is finitely $\Sigma-$graded (cf. lemma \ref{trans}). 
Let $I=\{1,2,3,4,5,6\}$,  then the root basis $\Pi_a$ of $\Sigma$ is adapted to the root basis $\Pi$ of $\Delta$ and we have  $I_{re}= \{1,2, 5,6\}$ (not connected), $\Gamma_1=\{1,5\}$, $\Gamma_2=\{2,6\}$, $J=\{4\}$, $J_{re}=\emptyset$, $I'_{im}=\{3\}$ and $J^{\circ}=J=\{4\}$.  
\\
Note that, for this example,  $\mathfrak{g}(I_{re})$, which is $\Sigma-$graded,  is isomorphic to $\mathfrak{m}\times \mathfrak{m}$. 
This gradation corresponds to that of the pseudo-complex real form  of $\mathfrak{m}\times \mathfrak{m}$ (i.e. the complex Kac-Moody algebra $\mathfrak{m}$ viewed as real Lie algebra) by its restricted reduced root system. Since the pair $(I_3,J_3)=(\{3,4\},\{4\})$  is not admissible, it is not possible to bring back $J^{\circ}$ to the empty set i.e. to build a Kac-Moody algebra $\mathfrak g^J$ grading finitely $\mathfrak g$ and maximally finitely $\Sigma-$graded.

\end{document}